\documentclass{amsart}

\usepackage{kotex}
\usepackage{amsmath}
\usepackage{mathrsfs}
\usepackage{amsfonts}
\usepackage{amscd}
\usepackage{slashed}
\usepackage{stmaryrd}

\usepackage{amssymb}
\usepackage{graphicx}
\usepackage{caption}
\usepackage{subcaption}
\usepackage{dsfont}
\usepackage{color}
\usepackage{hyperref}
\usepackage{bbm}
\usepackage{a4wide}
\usepackage{sseq}
\usepackage{tikz-cd}
\usepackage[abs]{overpic} 

\newtheorem{theorem}{Theorem}[section]
\newtheorem{lemma}[theorem]{Lemma}
\newtheorem{proposition}[theorem]{Proposition}
\newtheorem{corollary}[theorem]{Corollary}

\newtheorem{assumption}[theorem]{Assumption}
\newtheorem{example}[theorem]{Example}

\theoremstyle{definition}
\newtheorem*{definition*}{Definition}
\newtheorem{definition}[theorem]{Definition}

\theoremstyle{remark}
\newtheorem*{remark*}{Remark}
\newtheorem{remark}[theorem]{Remark}

\theoremstyle{question}
\newtheorem*{question*}{Question}

\numberwithin{equation}{section}

\newcommand{\myproof}[2]{Proof of {#1} {#2}}

\newcommand{\tripleslash}{/\!\!/\!\!/}

\begin{document}

\title[Quantization commutes with reduction for coisotropic A-branes]{Quantization commutes with reduction\\
	for coisotropic A-branes}

\author{NaiChung Conan Leung, Ying Xie and YuTung Yau}
\address{The Institute of Mathematical Sciences and Department of Mathematics, The Chinese University of Hong Kong, Shatin, Hong Kong}
\email{leung@ims.cuhk.edu.hk}
\address{School of Mathematical Sciences, Shenzhen University, Shenzhen, 518061, P.R. China}
\email{xieying@szu.edu.cn}
\address{Kavli Institute for the Physics and Mathematics of the Universe (WPI), The University of Tokyo Institutes for Advanced Study, The University of Tokyo, Kashiwa, Chiba 277-8583, Japan}
\email{yu-tung.yau@ipmu.jp}

\thanks{}

\maketitle


\begin{abstract}
	On a Hamiltonian $G$-manifold $X$, we define the notion of $G$-invariance of coisotropic A-branes $\mathcal{B}$. Under neat assumptions, we give a Marsden-Weinstein-Meyer type construction of a coisotropic A-brane $\mathcal{B}_{\operatorname{red}}$ on $X \sslash G$ from $\mathcal{B}$, recovering the usual construction when $\mathcal{B}$ is Lagrangian. For a canonical coisotropic A-brane $\mathcal{B}_{\operatorname{cc}}$ on a holomorphic Hamiltonian $G_\mathbb{C}$-manifold $X$, there is a fibration of $(\mathcal{B}_{\operatorname{cc}})_{\operatorname{red}}$ over $X \sslash G_\mathbb{C}$.\par
	We also show that `intersections of A-branes commute with reduction'. When $X = T^*M$ for $M$ being compact K\"ahler with a Hamiltonian $G$-action, Guillemin-Sternberg `quantization commutes with reduction' theorem can be interpreted as $\operatorname{Hom}_{X \sslash G}(\mathcal{B}_{\operatorname{red}}, (\mathcal{B}_{\operatorname{cc}})_{\operatorname{red}}) \cong \operatorname{Hom}_X(\mathcal{B}, \mathcal{B}_{\operatorname{cc}})^G$ with $\mathcal{B} = M$.
\end{abstract}

\section{Introduction}
Since the birth of Marsden-Weinstein-Meyer reduction \cite{MarWei1974, Mey1973} of Hamiltonian manifolds in 1970s, reductions of this type have appeared in numerous kinds of geometry (e.g. \cite{Alv2023, EchMunRom1999, Gin2012, GuiSte1982, HitKarLinRov1987, Kir1984, MarRat1986}). Recently, there have come various developments of equivariant Lagrangian Floer theory (e.g. \cite{HenLipSar2020, KimLauZhe2023, LauLeuLi2023, LekSeg2023}) for quantum cohomologies and Fukaya categories of symplectic reductions.\par
Motivated by the Homological Mirror Symmetry Conjecture \cite{Kon1995}, Kapustin-Orlov \cite{KapOrl2003} argued that \emph{coisotropic A-branes} $\mathcal{B} = (C, E, \nabla)$ should be included in the Fukaya category. Recall that $(E, \nabla)$ is a Hermitian vector bundle over a coisotropic submanifold $C$ of $X$ so that these structures induce a transverse holomorphic symplectic form on $C$ (see also \cite{Gua2011, Her2012}). There are very few examples of coisotropic A-branes, but they play important roles in mirror symmetry \cite{AldZas2005, ChaLeuZha2018, KapOrl2003, KapOrl2004, KatSor2018}, quantum mechanics \cite{GukWit2009}, geometric representation theory \cite{GukKorNawPeiSab2023}, (analytic) geometric Langlands program \cite{EtiFreKaz2021, GaiWit2022, KapWit2007}, knot theory \cite{AgaCosMcnVaf2017, AhaFelHon2019}, etc.\par
Throughout this paper, unless otherwise specified, we have the following standing assumption. 

\begin{assumption}
	$(X, \omega)$ is a symplectic manifold with a Hamiltonian $G$-action $\rho$ by a compact connected Lie group $G$ and a moment map $\mu: X \to \mathfrak{g}^*$ such that $G$ acts on $\mu^{-1}(0)$ freely.
\end{assumption}

We introduce the notion of $G$\emph{-invariance} of coisotropic A-branes $\mathcal{B} = (C, E, \nabla)$, namely $C$ is $G$-invariant and $(E, \nabla)$ is $G$-equivariant. These structures induce a section $\mu_\mathcal{B} \in \Gamma(C, \mathfrak{g}^* \otimes \operatorname{End}(E))^G$, which we call the \emph{moment section} of $\mathcal{B}$. Explicitly,
\begin{equation}
	\label{key}
	\langle \mu_\mathcal{B}, a \rangle = \mathcal{L}_a^E - \nabla_{\chi(a)} + \langle \mu \vert_C, a \rangle \cdot \operatorname{Id}_E
\end{equation}
for $a \in \mathfrak{g}$ (see Definition \ref{Definition 3.4}). Taking a normalization of the trace of $\mu_\mathcal{B}$, we obtain a $G$-equivariant map $\nu: C \to \mathfrak{g}^* \otimes \mathbb{C}$, which is a constant map when $C$ is Lagrangian; and a holomorphic moment map when $\mathcal{B}$ is a canonical coisotropic A-brane on a holomorphic Hamiltonian $G_\mathbb{C}$-manifold $X$. In general, by \eqref{Equation 1.1}, the imaginary part of $\nu$ coincides with the restriction of the moment map $\mu: X \to \mathfrak{g}^*$ onto $C$. It implies that the level set $\nu^{-1}(0)$ lies in $\mu^{-1}(0)$, and hence the quotient $C_{\operatorname{red}} := \nu^{-1}(0)/G$ is contained in the symplectic reduction $X \sslash G := \mu^{-1}(0)/G$.\par
We define the notion of \emph{neatness} of $\mu_\mathcal{B}$, which is a higher rank generalization of cleanness of the moment map value $0$. For $E$ of rank-$1$, $\mu_\mathcal{B}$ is neat if and only if $0$ is a clean value of $\nu$. Our main result shows that, under neat assumptions, $(E, \nabla)$ descends to a pair $(E_{\operatorname{red}}, \nabla_{\operatorname{red}})$ on $C_{\operatorname{red}}$ and $\mathcal{B}_{\operatorname{red}} := (C_{\operatorname{red}}, E_{\operatorname{red}}, \nabla_{\operatorname{red}})$ is a coisotropic A-brane on $X \sslash G$. We call $\mathcal{B}_{\operatorname{red}}$ the \emph{brane reduction} of $\mathcal{B}$.

\begin{theorem}[$=$ Theorem \ref{Theorem 4.11}]
	\label{Theorem 1.2}
	Let $\mathcal{B}$ be a $G$-invariant coisotropic A-brane on $(X, \omega)$. If the moment section $\mu_\mathcal{B}$ of $\mathcal{B}$ is neat, then $\mathcal{B}_{\operatorname{red}}$ is a coisotropic A-brane on the symplectic reduction $(X \sslash G, \omega_{\operatorname{red}})$.
\end{theorem}

When $C$ is Lagrangian in $X$, $C_{\operatorname{red}}$ recovers $C/G$ as a Lagrangian submanifold of $X \sslash G$. When $\mathcal{B}$ is a $G$-invariant canonical coisotropic A-brane on a holomorphic Hamiltonian $G_\mathbb{C}$-manifold $X$, $C_{\operatorname{red}}$ is heuristically thought of as a fibration over the holomorphic symplectic reduction of $X$ (see Subsection \ref{Subsection 4.02} for a more precise meaning). For example, take the holomorphic symplectic manifold $X = T^*(\mathbb{C}^2 \backslash 0)$ with the holomorphic Hamiltonian $\mathbb{C}^*$-action induced by the diagonal $\mathbb{C}^*$-action on $\mathbb{C}^2 \backslash 0$. Then $X$ admits an $\mathbb{S}^1$-invariant canonical coisotropic A-brane such that $C_{\operatorname{red}} \cong \mathbb{R} \times T^*\mathbb{P}^1$ is a trivial real line bundle over a holomorphic symplectic manifold $T^*\mathbb{P}^1$.\par
Motivated by equivariant Lagrangian Floer theory and Teleman Conjecture \cite{Tel2014}, we will examine the following principle in this paper:
\begin{center}
	\emph{`A-model morphism spaces commute with brane reduction'},
\end{center}
while a related statement that `intersections of A-branes \footnote{By the intersection of A-branes, we mean the intersection of their supports.} commute with brane reduction' will also be proved in Section \ref{Section 6}.\par
Unfortunately, the above principle is not phrased in a fully mathematical rigour because, for a \emph{general} pair of coisotropic A-branes $\mathcal{B}, \mathcal{B}'$ on $(X, \omega)$, the A-model morphism space $\operatorname{Hom}_X(\mathcal{B}, \mathcal{B}')$ has no well established mathematical definition in the existing literature. When $\mathcal{B}, \mathcal{B}'$ are Lagrangian A-branes so that $\operatorname{Hom}_X(\mathcal{B}, \mathcal{B}')$ can be mathematically realized as the Floer cohomology between them, the above principle has been studied by \cite{HenLipSar2020, KimLauZhe2023, LauLeuLi2023, LekSeg2023}, etc.\par
Beyond the Lagrangian case, attempts to describe $\operatorname{Hom}_X(\mathcal{B}, \mathcal{B}')$ mathematically can be found in \cite{AldZas2005, BisGua2022, KapOrl2004, LeuYau2024, Qin2020, Qiu2024, Qui2012, Yau2024}, etc. In \cite{GukWit2009} where they gave a physical proposal called \emph{brane quantization}, Gukov-Witten investigated the case when $(X, \omega)$ admits a space filling A-brane $\mathcal{B}_{\operatorname{cc}} = (X, L, \nabla)$ and a Lagrangian A-brane $\mathcal{B} = (M, M \times \mathbb{C}, d)$ such that $(L, \nabla)$ restricts to a Hermitian holomorphic line bundle $(L_M, \nabla^{L_M})$ over $M$ and the curvature $-\sqrt{-1} \omega_M$ of $\nabla^{L_M}$ is either non-degenerate or zero. They gave a physical argument that $\operatorname{Hom}_X(\mathcal{B}, \mathcal{B}_{\operatorname{cc}})$ should be regarded as the sheaf cohomology of holomorphic sections of $L_M$, which is a geometric quantization of $(M, \omega_M)$ when $\omega_M$ is non-degenerate (see Subsection \ref{Subsection 5.1} for more details). In this paper, we will take such a pair $(\mathcal{B}, \mathcal{B}_{\operatorname{cc}})$ of A-branes as a testing example of the above principle. This is briefly illustrated as follows.

\subsection{Guillemin-Sternberg theorem as $\operatorname{Hom}_{X \sslash G}(\mathcal{B}_{\operatorname{red}}, (\mathcal{B}_{\operatorname{cc}})_{\operatorname{red}}) \cong \operatorname{Hom}_X(\mathcal{B}, \mathcal{B}_{\operatorname{cc}})^G$}
\quad\par
When $M$ is a compact K\"ahler manifold equipped with a Hamiltonian $G$-action and a $G$-equivariant prequantum line bundle $(L_M, \nabla^{L_M})$, $X = T^*M$ admits a $G$-invariant Lagrangian A-brane $\mathcal{B} = (M, M \times \mathbb{C}, d)$ and a $G$-invariant space-filling A-brane $\mathcal{B}_{\operatorname{cc}} = (X, L, \nabla)$ such that $(L \vert_M, \nabla \vert_M) = (L_M, \nabla^{L_M})$. In alignment with Gukov-Witten brane quantization \cite{GukWit2009}, we give a new perspective on (a refinement \cite{Bra2001, Tel2000} of) Guillemin-Sternberg `quantization commutes with reduction' theorem \cite{GuiSte1982}, which states that
\begin{equation*}
	H_{\overline{\partial}}^{0, *}(M \sslash G, L_{M \sslash G}) \cong H_{\overline{\partial}}^{0, *}(M, L_M)^G.
\end{equation*}
The above theorem can be interpreted as
\begin{equation}
	\label{Equation 1.1}
	\operatorname{Hom}_{X \sslash G}(\mathcal{B}_{\operatorname{red}}, (\mathcal{B}_{\operatorname{cc}})_{\operatorname{red}}) \cong \operatorname{Hom}_X(\mathcal{B}, \mathcal{B}_{\operatorname{cc}})^G.
\end{equation}
A precise statement corresponding to \eqref{Equation 1.1} will appear in Proposition \ref{Proposition 7.11} as one of our main results. Note that this interpretation comes from the perspective of A-models, rather than B-models.

\subsection{Equivariant Dolbeault cohomology as $\operatorname{Hom}_{X \sslash G}(\mathcal{B}_{\operatorname{red}}, (\mathcal{B}_{\operatorname{cc}})_{\operatorname{red}}) \cong \operatorname{Hom}_X(\mathcal{B}, \mathcal{B}_{\operatorname{cc}})^G$}
\quad\par
Alternatively, when $M$ is a complex manifold with a free holomorphic $G_\mathbb{C}$-action, $X = T^*M$ admits a $G$-invariant Lagrangian A-brane $\mathcal{B} = (M, M \times \mathbb{C}, d)$ and a $G$-invariant space-filling A-brane $\mathcal{B}_{\operatorname{cc}} = (X, L, \nabla)$ such that $(L \vert_M, \nabla \vert_M) = (M \times \mathbb{C}, d)$. The quotient manifold $M_0 := M/G$ has a transverse holomorphic structure $\mathcal{E}_0$ induced by the canonical projection $M_0 \to M/G_\mathbb{C}$. We will show that there is an isomorphism
\begin{equation*}
	H_{\mathcal{E}_0}^*(M_0) \cong H_{\overline{\partial}}^{0, *}(M)^G
\end{equation*}
between the Chevalley-Eilenberg cohomology of $\mathcal{E}_0$ and the $G$-equivariant Dolbeault cohomology of $M$, and it can be interpreted as (\ref{Equation 1.1}) in this case.\par
\quad\par
As a final remark, the support $C$ of a coisotropic A-brane on $(X, \omega)$ is precisely a generalized complex submanifold of $(X, \mathcal{J}_\omega)$, where $\mathcal{J}_\omega$ is the generalized complex structure induced by $\omega$. While there is a vast amount of work on reduction of generalized complex manifolds (e.g. \cite{BurCavGua2007, Hu2009, LinTol2006, StiXu2008, Vai2007}), Zambon \cite{Zam2008} also worked on reduction of generalized complex submanifolds. Yet, our construction is different from his -- the brane reduction of $C$ in his sense is the coisotropic reduction of $C$, which is in general not equal to $C_{\operatorname{red}}$ given in Theorem \ref{Theorem 1.2}.\par
Our paper is organized as follows. Section \ref{Section 2} is a review on the definition of coisotropic A-branes. In Section \ref{Section 3}, we will define $G$-invariance of coisotropic A-branes. Section \ref{Section 4} provides a proof of Theorem \ref{Theorem 1.2}. In Section \ref{Section 5}, we will discuss the relationship between A-model morphism spaces and sheaf cohomologies on intersections of A-branes. In Section \ref{Section 6}, we will show that `intersections of A-branes commute with brane reduction'. In Section \ref{Section 7}, we will suggest a new perspective on equivariant Dolbeault cohomology and Guillemin-Sternberg theorem \cite{GuiSte1982} via sheaf cohomologies on intersections of $G$-invariant A-branes on $T^*M$ for a $G$-manifold $M$.

\subsection*{Acknowledgement}
\quad\par
The first author was substantially supported by grants from the Research Grants Council of the Hong Kong Special Administrative Region, China (Project No. CUHK14306322, CUHK14305923 and CUHK14302224), National Science Foundation of China (Project No. NSFC12341101), and direct grants from the Chinese University of Hong Kong. The second author thanks Ki Fung Chan, Yalong Cao, Ziming Nikolas Ma for helpful discussions and general advice related to this work. The third author thanks Ping Xu for his lecture in the NCTS Workwhop on Groupoids and Quantization, in which he mentioned a stacky approach to reduction. The third author also thanks Yan-Lung Leon Li, from whom he learnt about recent developments in equivariant Floer theory. The authors are grateful to the referee for valuable and constructive comments.

\section{Coisotropic A-branes}
\label{Section 2}
Coisotropic A-branes were discovered by Kapustin-Orlov \cite{KapOrl2003} in the rank-$1$ case, and were then formulated in terms of generalized complex geometry by Gualtieri \cite{Gua2011}. A higher rank definition was later introduced by Herbst \cite{Her2012}. The goal of this section is to recall these definitions -- the rank-$1$ case in Subsection \ref{Subsection 2.1} and the higher rank case in Subsection \ref{Subsection 2.2}. Throughout this section, we do not require that the symplectic manifold $(X, \omega)$ is equipped with a Hamiltonian $G$-action.

\subsection{Rank-$1$ coisotropic A-branes}
\label{Subsection 2.1}
\quad\par
As we will see, every coisotropic A-brane has an underlying transverse holomorphic symplectic manifold. Now, We first recall the definition of transverse holomorphic symplectic forms. Let $C$ be a smooth manifold. For a complex valued $2$-form $\Omega$ on $C$, when $(TC)^{\perp \Omega}$ is a vector subbundle of $TC$, we denote by $T_{\operatorname{quo}} C$ the quotient bundle $TC / (TC)^{\perp \Omega}$ and by $\widetilde{\Omega}$ the $2$-form on $T_{\operatorname{quo}}C$ descended by $\Omega$ (see Appendix \ref{Appendix A} for the linear algebra involved).

\begin{definition}
	\label{Definition 2.1}
	A \emph{transverse holomorphic symplectic form} on $C$ is a closed $\mathbb{C}$-valued $2$-form $\Omega$ on it such that $(TC)^{\perp \Omega} \subset TC$ is a vector subbundle and the endomorphism $I: T_{\operatorname{quo}} C \to T_{\operatorname{quo}} C$ defined by the following condition is an almost complex structure on $T_{\operatorname{quo}} C$:
	\begin{equation*}
		\operatorname{Re} \widetilde{\Omega}(u, v) = \operatorname{Im} \widetilde{\Omega}(Iu, v), \quad \text{for all } u, v \in \Gamma(C, T_{\operatorname{quo}} C).
	\end{equation*}
\end{definition}

\begin{remark}
	\label{Remark 2.2}
	Because $d\Omega = 0$, the following complex vector subbundle $\mathcal{E}$ of $TC \otimes \mathbb{C}$ is involutive:
	\begin{equation}
		\label{Equation 2.1}
		\mathcal{E}_x := \{ u \in T_xC \otimes \mathbb{C}: \Omega(u, v) = 0 \text{ for all } v \in T_xC \otimes \mathbb{C} \}
	\end{equation}
	for all $x \in C$. Note that $\mathcal{E}$ coincides with the kernel of the canonical projection $TC \otimes \mathbb{C} \to T_{\operatorname{quo}}^{1, 0}C$. Hence, $(TC)^{\perp \Omega} = \mathcal{E} \cap \overline{\mathcal{E}} \cap TC$ is involutive and $I$ is integrable (see \cite{Dor1993, KapOrl2003} and Appendix \ref{Appendix B}).
\end{remark}

\begin{remark}
	Writing $\Omega = F + \sqrt{-1} \omega$ for real $2$-forms $F$ and $\omega$ on $C$, the condition on the almost complex structure $I$ on $T_{\operatorname{quo}}C$ implies that the kernels of $F$ and of $\omega$ are both equal to $(TC)^{\perp \Omega}$, and $F$ and $\omega$ descend to $\operatorname{Re} \widetilde{\Omega}$ and $\operatorname{Im} \widetilde{\Omega}$ respectively.
\end{remark}

Here and in the sequel, by saying that $(E, \nabla)$ is a Hermitian vector bundle over $C$, we mean that $E$ is a Hermitian vector bundle over $C$ and $\nabla$ is a unitary connection on $E$. We also denote the curvature of $\nabla$ by $F^\nabla$. Here comes the desired definition in the rank-$1$ case.

\begin{definition}[\cite{Gua2011, KapOrl2003}]
	A \emph{rank-}$1$ \emph{coisotropic A-brane} on $(X, \omega)$ is a triple
	\begin{equation*}
		\mathcal{B} = (C, E, \nabla)
	\end{equation*}
	of a coisotropic submanifold $C$ of $(X, \omega)$ and a Hermitian line bundle $(E, \nabla)$ over $C$ such that $\Omega := F + \sqrt{-1} \omega \vert_C$ is a transverse holomorphic symplectic form on $C$, where $F = \sqrt{-1}F^\nabla$.
\end{definition}

As we will see in the next subsection, to generalize this definition to the higher rank case, we need to recall notions on top of the underlying transverse holomorphic symplectic geometry.

\subsection{Higher rank coisotropic A-branes}
\label{Subsection 2.2}
\quad\par
Consider a transverse holomorphic symplectic manifold $(C, \Omega)$ and a Hermitian vector bundle $(E, \nabla)$ of rank $r$ over $C$. We decompose the curvature $F^\nabla$ into two components:
\begin{equation*}
	F^\nabla = F_0^\nabla - \sqrt{-1}F_{\operatorname{tr}}^\nabla \cdot \operatorname{Id}_E.
\end{equation*}
where $F_{\operatorname{tr}}^\nabla$ is the real valued $2$-form on $C$ defined by $F_{\operatorname{tr}}^\nabla = \tfrac{\sqrt{-1}}{r} \operatorname{Tr} F^\nabla$ and $F_0^\nabla \in \Omega^2(C, \operatorname{End}(E))$ is such that $\operatorname{Tr} F_0^\nabla = 0$. We call $F_0^\nabla$ the \emph{traceless curvature} of $\nabla$ (see, for instance, \cite{Ver1996}). Heuristically, $F_0^\nabla$ is the curvature of $E \otimes (\sqrt[r]{\det E})^*$, though an $r$th root $\sqrt[r]{\det E}$ of $\det E$ might not exist as a Hermitian line bundle. When $r = 1$, $F_0^\nabla = 0$.\par
Let $I$ be the integrable almost complex structure on $T_{\operatorname{quo}}C$ associated with $(C, \Omega)$. We say that an $E$-valued $2$-form $\beta \in \Omega^2(C, E)$ is \emph{of transverse type} $(1, 1)$ if it descends to a smooth section $\widetilde{\beta} \in \Gamma\left( C, \textstyle\bigwedge^2 T_{\operatorname{quo}}^*C \otimes E \right)$ such that $\widetilde{\beta}(Iu, Iv) = \widetilde{\beta}(u, v)$ for all $u, v \in \Gamma(C, T_{\operatorname{quo}}C)$.\par
We are now ready to state the definition of coisotropic A-branes (of higher ranks).

\begin{definition}[\cite{Her2012}]
	\label{Definition 2.4}
	A \emph{coisotropic A-brane} \footnote{Even more generally, Herbst \cite{Her2012} described these branes as pairs of a coisotropic submanifold and a complex of vector bundles, each of which bundle is as in Definition \ref{Definition 2.4}, with a non-commutative differential.} (or \emph{A-brane} in short) on $(X, \omega)$ is a triple
	\begin{equation*}
		\mathcal{B} = (C, E, \nabla)
	\end{equation*}
	of a coisotropic submanifold $C$ of $(X, \omega)$ and a Hermitian vector bundle $(E, \nabla)$ over $C$ such that
	\begin{enumerate}
		\item $\Omega := F_{\operatorname{tr}}^\nabla + \sqrt{-1} \omega \vert_C$ is a transverse holomorphic symplectic form on $C$; and
		\item $F_0^\nabla$ is of transverse type $(1, 1)$.
	\end{enumerate}
\end{definition}

We call $C$ the \emph{support} of $\mathcal{B}$.

\begin{remark}
	\label{Remark 2.5}
	Observe that $F_0^\nabla$, but not $F^\nabla$ in general, vanishes on $\mathcal{E}$, where $\mathcal{E}$ is defined as in \eqref{Equation 2.1}. It implies that the following $\operatorname{End}(E)$-valued $2$-form vanishes on $\mathcal{E}$:
	\begin{equation}
		F^\nabla + (\omega \vert_C) \operatorname{Id}_E = F_0^\nabla - \sqrt{-1}\Omega \operatorname{Id}_E.
	\end{equation}
\end{remark}

At the end of this section, we will go through some examples.

\begin{example}
	Let $C$ be a Lagrangian submanifold of $(X, \omega)$ and $(E, \nabla)$ be a flat Hermitian vector bundle over $C$. Then $(C, E, \nabla)$ is an A-brane on $(X, \omega)$, known as a \emph{Lagrangian A-brane}. Indeed, if the support of an A-brane is Lagrangian, then the A-brane must be of the above form.
\end{example}

\begin{example}
	Let $(C, L, \nabla^L)$ be a rank-$1$ A-brane on $(X, \omega)$ and $(E, \nabla^E)$ be a flat Hermitian vector bundle over $C$. Then $(E \otimes L, \nabla^{E \otimes L})$ is projectively flat and $(C, E \otimes L, \nabla^{E \otimes L})$ is an A-brane on $(X, \omega)$. Examples of such branes appeared in \cite{ChaLeuZha2018}.
\end{example}

\begin{example}
	Suppose that $X$ admits a hyperK\"ahler structure defined by complex structures $J, K$ with $JK = -KJ$ and their respective K\"ahler forms $\omega_J, \omega_K$ such that $\omega = \omega_K$. Let $(E, \nabla)$ be a projectively hyperholomorphic bundle over $X$ (see \cite{Ver1996}) such that $F_{\operatorname{tr}}^\nabla = \omega_J$. Then $(X, E, \nabla)$ is an A-brane on $(X, \omega)$. Examples of such branes appeared in \cite{GukKorNawPeiSab2023, GukWit2009, KapWit2007}.
\end{example}

\section{$G$-invariant coisotropic A-branes}
\label{Section 3}
The goal of this section is to introduce the notion of $G$\emph{-invariance} of (coisotropic) A-branes $\mathcal{B} = (C, E, \nabla)$ on a Hamiltonian $G$-manifold. This notion will be given in Subsection \ref{Subsection 3.3}. As a preparation, we will discuss Hamiltonian $G$-actions on $C$ in Subsection \ref{Subsection 3.1}, and $G$-equivariant structures on $(E, \nabla)$ in Subsection \ref{Subsection 3.2}.

\subsection{Hamiltonian actions on transverse holomorphic symplectic manifolds}
\label{Subsection 3.1}
\quad\par
Let $C$ be a smooth $G$-manifold and $\chi_C$ be the induced infinitesimal $\mathfrak{g}$-action on it.

\begin{definition}
	\label{Definition 3.1}
	Let $\Omega$ be a transverse holomorphic symplectic form on $C$. The $G$-action on $C$ is said to be \emph{Hamiltonian} on $(C, \Omega)$  if $\Omega$ is $G$-invariant and there exists a $G$-equivariant map $\nu: C \to \mathfrak{g}^* \otimes \mathbb{C}$, called a \emph{moment map}, such that for all $a \in \mathfrak{g}$ and $u \in \Gamma(C, TC)$,
	\begin{equation}
		\Omega(\chi_C(a), u) = \langle d\nu(u), a \rangle.
	\end{equation}
\end{definition}

In the above definition, we do not require that the $G$-action on $C$ extends to a $G_\mathbb{C}$-action. Still, we can deduce that there is an infinitesimal $\mathfrak{g}_\mathbb{C}$-action on  $T_{\operatorname{quo}}C = TC / (TC)^{\perp \Omega}$ preserving the integrable almost complex structure $I$ on it (see Appendix \ref{Appendix B}).

\subsection{Equivariant Hermitian vector bundles and Hamiltonian actions}
\label{Subsection 3.2}
\quad\par
Suppose furthermore that $(E, \nabla)$ is a $G$-\emph{equivariant} Hermitian vector bundle over $C$, i.e. the underlying complex vector bundle $E$ is $G$-equivariant and both the Hermitian metric on $E$ and $\nabla$ are $G$-invariant. Then there is an induced action $\mathcal{L}_a^E$ of $a \in \mathfrak{g}$ on $\Gamma(C, E)$. We can then define a smooth section $\mu^\nabla \in \Gamma(C, \mathfrak{g}^* \otimes \operatorname{End}(E))$ by
\begin{equation}
	\langle \mu^\nabla, a \rangle(s) := \mu^\nabla(a \otimes s) = \mathcal{L}_a^Es - \nabla_{\chi_C(a)} s.
\end{equation}
for all $a \in \mathfrak{g}$ and $s \in \Gamma(C, E)$. It is a generalization of moment maps in the following sense.

\begin{proposition}
	\label{Proposition 3.2}
	$\mu^\nabla \in \Gamma(C, \mathfrak{g}^* \otimes \mathfrak{u}(E))^G$. Moreover, for all $a, b \in \mathfrak{g}$, $\iota_{\chi_C(a)}F^\nabla = \nabla \langle \mu^\nabla, a \rangle$ and
	\begin{equation}
		\label{Equation 3.3}
		F^\nabla(\chi_C(a), \chi_C(b)) = \langle \mu^\nabla, a \rangle \circ \langle \mu^\nabla, b \rangle - \langle \mu^\nabla, b \rangle \circ \langle \mu^\nabla, a \rangle - \langle \mu^\nabla, [a, b] \rangle.
	\end{equation}
\end{proposition}
\begin{proof}
	Let $h$ denote the Hermitian metric on $E$. For $a \in \mathfrak{g}$ and $s, s' \in \Gamma(C, E)$, since both
	\begin{equation*}
		h(\mathcal{L}_a^Es, s') + h(s, \mathcal{L}_a^E s') \quad \text{and} \quad h(\nabla_{\chi_C(a)} s, s') + h(s, \nabla_{\chi_C(a)} s')
	\end{equation*}
	are equal to $\mathcal{L}_{\chi_C(a)} (h(s, s'))$, we have $h(\langle \mu^\nabla, a \rangle(s), s') + h(s, \langle \mu^\nabla, a \rangle(s')) = 0$. By $G$-invariance of $\nabla$ and the fundamental properties of the infinitesimal $\mathfrak{g}$-actions $\mathcal{L}^E$ on $\Gamma(C, E)$ and $\chi_C$ on $C$, we can easily deduce that the section $\mu^\nabla$ of $\mathfrak{g}^* \otimes \mathfrak{u}(E)$ is $G$-invariant, where $G$ acts on $\mathfrak{g}^*$ by the coadjoint action. Now, let $a \in \mathfrak{g}$ and $v \in \Gamma(C, TC)$. By $G$-invariance of $\nabla$ again,
	\begin{equation*}
		\mathcal{L}_a^E \circ \nabla_v = \nabla_{[\chi_C(a), v]} + \nabla_v \circ \mathcal{L}_a^E.
	\end{equation*}
	Using this equality and the formula $\langle \mu^\nabla, a \rangle = \mathcal{L}_a^E - \nabla_{\chi_C(a)}$, we obtain
	\begin{equation*}
		F^\nabla(\chi_C(a), v)s = \nabla_v (\langle \mu^\nabla, a \rangle (s)) - \langle \mu^\nabla, a \rangle (\nabla_v s) = (\nabla_v\langle \mu^\nabla, a \rangle)(s),
	\end{equation*}
	for all $s \in \Gamma(C, E)$. Setting $v = \chi_C(b)$ for $b \in \mathfrak{g}$ yields
	\begin{align*}
		F^\nabla(\chi_C(a), \chi_C(b)) = & \nabla_{\chi_C(b)} \circ \langle \mu^\nabla, a \rangle - \langle \mu^\nabla, a \rangle \circ \nabla_{\chi_C(b)}\\
		= & \langle \mu^\nabla, a \rangle \circ \langle \mu^\nabla, b \rangle - \langle \mu^\nabla, b \rangle \circ \langle \mu^\nabla, a \rangle + \mathcal{L}_b^E \circ \langle \mu^\nabla, a \rangle - \langle \mu^\nabla, a \rangle \circ \mathcal{L}_b^E.
	\end{align*}
	Passing $G$-invariance of $\mu^\nabla$ to the infinitesimal level yields
	\begin{align*}
		\mathcal{L}_b^E \circ \langle \mu^\nabla, a \rangle = \langle \mu^\nabla, [b, a] \rangle + \langle \mu^\nabla, a \rangle \circ \mathcal{L}_b^E.
	\end{align*}
	Consequently, \eqref{Equation 3.3} holds.
\end{proof}

Recall that $F_{\operatorname{tr}}^\nabla = \tfrac{\sqrt{-1}}{r} \operatorname{Tr} F^\nabla$ is a $d$-closed $\mathbb{R}$-valued $2$-form on $C$, where $r$ is the rank of $E$. Indeed, the fact that $(E, \nabla)$ is $G$-equivariant implies that the $G$-action on $C$ is \emph{Hamiltonian} \footnote{In some literature (e.g. Definition 3 in \cite{EchMunRom1999}), when $F_{\operatorname{tr}}^\nabla$ is a presymplectic form, i.e. when $F_{\operatorname{tr}}^\nabla$ is of constant rank, it is called a strongly Hamiltonian action. However, we just call it a Hamiltonian action in this paper.} on $(C, F_{\operatorname{tr}}^\nabla)$ in the sense that $F_{\operatorname{tr}}^\nabla$ is $G$-invariant and there exists a $G$-equivariant smooth map $\mu_{\operatorname{tr}}^\nabla: C \to \mathfrak{g}^*$ ($\mu_{\operatorname{tr}}^\nabla = \tfrac{\sqrt{-1}}{r} \operatorname{Tr} \mu^\nabla$ in this case) such that for all $a \in \mathfrak{g}$ and $u \in \Gamma(C, TC)$,
\begin{equation*}
	F_{\operatorname{tr}}^\nabla(\chi_C(a), u) = \langle d\mu_{\operatorname{tr}}^\nabla(u), a \rangle.
\end{equation*}

\subsection{Invariant coisotropic A-branes}
\label{Subsection 3.3}

\begin{definition}
	A (coisotropic) A-brane $\mathcal{B} = (C, E, \nabla)$ on $(X, \omega)$ is said to be $G$-\emph{invariant} if $C$ is a $G$-invariant submanifold of $X$ and $(E, \nabla)$ is $G$-equivariant.
\end{definition}

Let $\chi$ be the induced infinitesimal $\mathfrak{g}$-action on $X$. $G$-invariance of an A-brane induces a smooth section of $\mathfrak{g}^* \otimes \operatorname{End}(E)$ as follows.

\begin{definition}
	\label{Definition 3.4}
	The \emph{moment section} of a $G$-invariant A-brane $\mathcal{B} = (C, E, \nabla)$ on $(X, \omega)$ is a smooth section $\mu_\mathcal{B} \in \Gamma(C, \mathfrak{g}^* \otimes \operatorname{End}(E))$ defined as follows. For all $a \in \mathfrak{g}$,
	\begin{equation}
		\langle \mu_\mathcal{B}, a \rangle := \mathcal{L}_a^E - \nabla_{\chi(a)} + \langle \mu \vert_C, a \rangle \cdot \operatorname{Id}_E.
	\end{equation}
\end{definition}

\begin{proposition}
	Let $\mathcal{B} = (C, E, \nabla)$ be a $G$-invariant A-brane on $(X, \omega)$. Then
	\begin{enumerate}
		\item $F^\nabla + (\omega \vert_C) \operatorname{Id}_E$ is $G$-invariant;
		\item the moment section $\mu_\mathcal{B}$ of $\mathcal{B}$ is $G$-invariant; and
		\item for all $a \in \mathfrak{g}$, $\iota_{\chi(a)} (F^\nabla + (\omega \vert_C) \operatorname{Id}_E) = \nabla \langle \mu_\mathcal{B}, a \rangle$.
	\end{enumerate}
\end{proposition}

This proposition follows from Proposition \ref{Proposition 3.2} and the definitions of the Hamiltonian $G$-action on $(X, \omega)$ and its moment map $\mu: X \to \mathfrak{g}^*$; we omit the details of the proof.\par
There induces a natural $G$-equivariant structure on $\det E$ (equipped with the induced unitary connection $\nabla^{\det E}$). As a result, the following proposition holds.

\begin{proposition}
	\label{Proposition 3.6}
	Let $\mathcal{B} = (C, E, \nabla)$ be a $G$-invariant rank-$r$ A-brane on $(X, \omega)$. Then $G$ acts on the transverse holomorphic symplectic manifold $C$ by a Hamiltonian action with a moment map
	\begin{equation*}
		\nu: C \to \mathfrak{g}^* \otimes \mathbb{C},
	\end{equation*}
	which we call the \emph{moment trace} of $\mathcal{B}$, uniquely determined by the following conditions:
	\begin{enumerate}
		\item the real part $\operatorname{Re} \nu$ of $\nu$ satisfies $\mathcal{L}_a^{\det E} = \nabla_{\chi(a)}^{\det E} - \sqrt{-1} r \langle \operatorname{Re} \nu, a \rangle$ for all $a \in \mathfrak{g}$; and
		\item the imaginary part of $\nu$ is equal to the restriction $\mu \vert_C$ of the moment map $\mu: X \to \mathfrak{g}^*$.
	\end{enumerate}
	In particular,
	\begin{equation}
		\label{Equation 3.4}
		\nu = \tfrac{\sqrt{-1}}{r} \operatorname{Tr} \mu_\mathcal{B}.
	\end{equation}
\end{proposition}
\begin{proof}
	Applying Proposition \ref{Proposition 3.2} to $(\det E, \nabla^{\det E})$, we have shown that
	\begin{equation*}
		F^{\nabla^{\det E}} = \operatorname{Tr} F^\nabla = \tfrac{r}{\sqrt{-1}} F_{\operatorname{tr}}^\nabla
	\end{equation*}
	is $G$-invariant, the following map $\mu^{\nabla^{\det E}}: C \to \mathfrak{g}^* \otimes \mathfrak{u}(\det E) \cong \mathfrak{g}^* \otimes \mathbb{R} \sqrt{-1}$ is $G$-equivariant:
	\begin{equation*}
		\langle \mu^{\nabla^{\det E}}, a \rangle := \mathcal{L}_a^{\det E} - \nabla_{\chi(a)}^{\det E}, \quad \text{ for all } a \in \mathfrak{g},
	\end{equation*}
	and $\frac{r}{\sqrt{-1}} \iota_{\chi(a)} F_{\operatorname{tr}}^\nabla = \langle d\mu^{\nabla^{\det E}}, a \rangle$ for all $a \in \mathfrak{g}$. In particular, the transverse holomorphic symplectic form $\Omega = F_{\operatorname{tr}}^\nabla + \sqrt{-1} \omega \vert_C$ on $C$ associated with $\mathcal{B}$ is $G$-invariant, the map
	\begin{equation*}
		\nu = \left( \tfrac{\sqrt{-1}}{r} \mu^{\nabla^{\det E}} \right) + \sqrt{-1} \mu \vert_C: C \to \mathfrak{g}^* \otimes \mathbb{C}
	\end{equation*}
	is $G$-equivariant, and $\iota_{\chi(a)} \Omega = \langle d\nu, a \rangle$ for all $a \in \mathfrak{g}$. Therefore, The $G$-action on $(C, \Omega)$ is Hamiltonian. The equality $\nu = \tfrac{\sqrt{-1}}{r} \operatorname{Tr} \mu_\mathcal{B}$ follows immediately.
\end{proof}

\section{Brane reduction}
\label{Section 4}
In this key section of our paper, we will prove Theorem \ref{Theorem 1.2} that we can perform a Marsden-Weinstein-Meyer type reduction to obtain an A-brane $\mathcal{B}_{\operatorname{red}}$ on $X \sslash G$ from a $G$-invariant A-brane $\mathcal{B} = (C, E, \nabla)$ on $X$ satisfying neat assumptions. We call it the \emph{brane reduction} of $\mathcal{B}$.\par
The brane reduction $\mathcal{B}_{\operatorname{red}} = (C_{\operatorname{red}}, E_{\operatorname{red}}, \nabla_{\operatorname{red}})$ consists of two components: $C_{\operatorname{red}}$ and $(E_{\operatorname{red}}, \nabla_{\operatorname{red}})$. Indeed, the construction of $C_{\operatorname{red}}$ only relies on the Hamiltonian $G$-action on $C$ as a transverse holomorphic symplectic manifold, without the need to embed $C$ into $X$ as a coisotropic submanifold. Subsection \ref{Subsection 4.1} is devoted to this issue. In Subsection \ref{Subsection 4.02}, we will comment on the relationships among our construction, holomorphic symplectic reductions and hyperK\"ahler quotients. In Subsection \ref{Subsection 4.2}, we will explain how to construct $(E_{\operatorname{red}}, \nabla_{\operatorname{red}})$ from $(E, \nabla)$. Finally, we will complete our proof of Theorem \ref{Theorem 1.2} in Subsection \ref{Subsection 4.3}.

\subsection{Brane reduction of Hamiltonian transverse holomorphic symplectic manifolds}
\label{Subsection 4.1}
\quad\par
We will prove a `Marsden-Weinstein-Meyer Theorem' for transverse holomorphic symplectic manifolds with Hamiltonian $G$-actions. Let us first recall the following definition.

\begin{definition}
	Let $f: Y \to Z$ be a smooth map between two smooth manifolds $Y$ and $Z$. A point $z \in Z$ is said to be a \emph{clean value} (or \emph{weakly regular value}) of $f$ if $S := f^{-1}(z)$ is a smooth submanifold of $Y$ and for all $y \in S$, $T_yS$ is the kernel of $df_y: T_yY \to T_zZ$.
\end{definition}

Let $G$ act on a transverse holomorphic symplectic manifold $(C, \Omega)$ by a Hamiltonian action with a moment map $\nu: C \to \mathfrak{g}^* \otimes \mathbb{C}$. Now, we state the aforementioned theorem.

\begin{theorem}
	\label{Theorem 4.2}
	Suppose that $G$ acts on $\nu^{-1}(0)$ freely and $0$ is a clean value of $\nu$. Then $\Omega \vert_{\nu^{-1}(0)}$ descends to a transverse holomorphic symplectic form $\Omega_{\operatorname{red}}$ on $C_{\operatorname{red}} := \nu^{-1}(0) / G$.
\end{theorem}

We call $(C_{\operatorname{red}}, \Omega_{\operatorname{red}})$ the \emph{brane reduction} of $(C, \Omega)$.\par
Before proceeding to the proof, we comment on the real and imaginary parts of $\Omega$ to provide intuition for the theorem. Both $\operatorname{Re} \Omega$ and $\operatorname{Im} \Omega$ are $d$-closed real $2$-forms of constant rank. Furthermore, the $G$-action on $C$ is Hamiltonian with respect to both forms, with $\operatorname{Re} \nu$ and $\operatorname{Im} \nu$ as their respective moment maps. Under the hypotheses of \cite[Theorem 2]{EchMunRom1999} --- including the requirement that $G$ acts on $(\operatorname{Re} \nu)^{-1}(0)$ freely --- we obtain a smooth manifold $(\operatorname{Re} \nu)^{-1}(0) / G$ equipped with a real closed $2$-form $(\operatorname{Re} \Omega)_{\operatorname{red}}$ whose pullback via the quotient map equals the restriction of $\operatorname{Re} \Omega$ to $(\operatorname{Re} \nu)^{-1}(0)$. An analogous construction for $\operatorname{Im} \Omega$ yields a $d$-closed real $2$-form $(\operatorname{Im} \Omega)_{\operatorname{red}}$ on $(\operatorname{Im} \nu)^{-1}(0) / G$. By summing the restrictions of $(\operatorname{Re} \Omega)_{\operatorname{red}}$ and $\sqrt{-1} (\operatorname{Im} \Omega)_{\operatorname{red}}$ to the intersection
\begin{equation*}
	\nu^{-1}(0) / G = ( (\operatorname{Re} \nu)^{-1}(0) / G ) \cap ( (\operatorname{Im} \nu)^{-1}(0) / G ),
\end{equation*}
we define a $d$-closed complex $2$-form $\Omega_{\operatorname{red}}$. Theorem \ref{Theorem 4.2} then asserts that this complex $2$-form is a transverse holomorphic symplectic form on $C_{\operatorname{red}}$. Critically, however, Theorem \ref{Theorem 4.2} does not actually require the full set of assumptions from \cite[Theorem 2]{EchMunRom1999} for its validity.\par
We state a corollary about the dimension of $C_{\operatorname{red}}$ and the rank of $\Omega_{\operatorname{red}}$. Let $\mathcal{F} = (TC)^{\perp \Omega}$ and $I$ be the almost complex structure on $T_{\operatorname{quo}}C = TC / \mathcal{F}$ determined by $\Omega$. Define a $\mathbb{C}$-linear map $\widetilde{\chi}_C: \mathfrak{g}_\mathbb{C} \to \Gamma(C, T_{\operatorname{quo}}C)$ by
\begin{equation}
	\label{Equation 4.1}
	\widetilde{\chi}_C(a + \sqrt{-1} b)(x) = (\chi_C(a)(x) + \mathcal{F}_x) + I(\chi_C(b)(x) + \mathcal{F}_x)
\end{equation}
for all $a, b \in \mathfrak{g}$ and $x \in C$, where $\chi_C$ is the induced infinitesimal $\mathfrak{g}$-action on $C$. For the $G$-action on $C$, by its \emph{transverse complex rank at} $x \in C$, we mean the complex rank of the $\mathbb{C}$-linear map $\mathfrak{g}_\mathbb{C} \to T_{\operatorname{quo}, x}C$ given by $a \mapsto \widetilde{\chi}_C(a)(x)$.

\begin{corollary}
	\label{Corollary 4.3}
	Under the hypothesis of Theorem \ref{Theorem 4.2}, the $G$-action on $C$ is of constant transverse complex rank $r$ on $\nu^{-1}(0)$. Also, $\dim C_{\operatorname{red}} = \dim C - 2r - \dim G$ and $\operatorname{rank} \Omega_{\operatorname{red}} = \operatorname{rank} \Omega - 4r$. In particular, the following inequalities hold:
	\begin{align*}
		\dim C - 3\dim G \leq & \dim C_{\operatorname{red}} \leq \dim C - \dim G,\\
		\operatorname{rank} \Omega - 4 \dim G \leq & \operatorname{rank} \Omega_{\operatorname{red}} \leq \operatorname{rank} \Omega.
	\end{align*}
\end{corollary}

We will provide a proof of this corollary after the proof of Theorem \ref{Theorem 4.2}.

\begin{example}
	\label{Example 4.5}
	Consider $C = \{ (z^1, z^2, w^1, w^2) \in \mathbb{C}^4: (z^1, z^2) \neq (0, 0) \}$ with the holomorphic symplectic form $\Omega = dz^1 \wedge dw^2 + dz^2 \wedge dw^2$ and the holomorphic Hamiltonian $\mathbb{C}^*$-action given by
	\begin{equation*}
		\lambda \cdot (z^1, z^2, w^1, w^2) = (\lambda z^1, \lambda z^2, \lambda^{-1} w^1, \lambda^{-1} w^2).
	\end{equation*}
	The holomorphic moment map $\nu: C \to \mathbb{C}$ given by $(z^1, z^2, w^1, w^2) \mapsto \sqrt{-1} (z^1w^1 + z^2w^2)$ is a moment map in the sense of Definition \ref{Definition 3.1} with respect to the subaction of $G = \mathbb{S}^1$ on $C$. This $\mathbb{S}^1$-action is of transverse complex rank $1$ on $\nu^{-1}(0)$ and $C_{\operatorname{red}} = \nu^{-1}(0) / \mathbb{S}^1 \cong \mathbb{R} \times T^*\mathbb{P}^1$ is $5$-dimensional. Also, $\Omega_{\operatorname{red}}$ is the pullback of a holomorphic symplectic form on $\nu^{-1}(0) / \mathbb{C}^* \cong T^*\mathbb{P}^1$ along the projection $C_{\operatorname{red}} \to \nu^{-1}(0) / \mathbb{C}^*$. Hence, in this case we have
	\begin{equation*}
		\dim C_{\operatorname{red}} = \dim C - 3 \dim G \quad \text{and} \quad \operatorname{rank} \Omega_{\operatorname{red}} = \operatorname{rank} \Omega - 4 \dim G.
	\end{equation*}
\end{example}

Example \ref{Example 4.5} shows that, even when $(C, \Omega)$ is a holomorphic symplectic manifold, its brane reduction $(C_{\operatorname{red}}, \Omega_{\operatorname{red}})$ might not be a holomorphic symplectic manifold.

\begin{example}
	\label{Example 4.6}
	Let $(B, \Omega_B)$ be a transverse holomorphic symplectic manifold and $q: C \to B$ be a principal $G$-bundle. We convert the principal right $G$-action into a left $G$-action by defining $g \cdot x := x \cdot g^{-1}$ for $x \in C$ and $g \in G$. Then $\Omega := q^*\Omega_B$ is a transverse holomorphic symplectic form on $C$ and the left $G$-action is a Hamiltonian action on $(C, \Omega)$ with the constant zero map $\nu: C \to \mathfrak{g}^* \otimes \mathbb{C}$ as a moment map. This $G$-action on $\nu^{-1}(0) = C$ is of transverse complex rank $0$ and $(C_{\operatorname{red}}, \Omega_{\operatorname{red}}) = (B, \Omega_B)$, therefore
	\begin{equation*}
		\dim C_{\operatorname{red}} = \dim C - \dim G \quad \text{and} \quad \operatorname{rank} \Omega_{\operatorname{red}} = \operatorname{rank} \Omega.
	\end{equation*}
\end{example}

Two extreme cases of Example \ref{Example 4.6} are when $\Omega_B = 0$ and when $\Omega_B$ is non-degenerate. The latter case suggests that, even when $\Omega$ is not a holomorphic symplectic form on $C$, Theorem \ref{Theorem 4.2} can yield a holomorphic symplectic manifold $(C_{\operatorname{red}}, \Omega_{\operatorname{red}})$. This example also motivates us to first prove a special case of Theorem \ref{Theorem 4.2}, which is essential to our proof of the general statement.

\begin{lemma}
	\label{Lemma 4.06}
	Let $(C, \Omega)$ be a transverse holomorphic symplectic manifold and let $G$ act on $C$ freely such that the $2$-form $\Omega$ is basic. Then $\Omega$ descends to a transverse holomorphic symplectic form $\Omega_{C/G}$ on $C/G$.
\end{lemma}

\begin{remark}
	The hypothesis of Lemma \ref{Lemma 4.06} is equivalent to that $G$ acts on a transverse holomorphic symplectic manifold $(C, \Omega)$ by a Hamiltonian action with the constant zero map $\nu: C \to \mathfrak{g}^* \otimes \mathbb{C}$ as a moment map such that the hypothesis of Theorem \ref{Theorem 4.2} is satisfied.\par
	Under this hypothesis, it is clear that $\Omega$ descends to a complex $2$-form $\Omega_{C/G}$ on the smooth manifold $C/G$. Given the constant zero map $\nu$, we have $C = \nu^{-1}(0)$. Hence, we will denote $C/G$ by $C_{\operatorname{red}}$ and $\Omega_{C/G}$ by $\Omega_{\operatorname{red}}$ in the proof of Lemma \ref{Lemma 4.06}.
\end{remark}

\begin{proof}[\myproof{Lemma}{\ref{Lemma 4.06}}]
	As $d \Omega = 0$ and $q$ is a smooth submersion, $d\Omega_{\operatorname{red}} = 0$. It remains to show that $\Omega_{\operatorname{red}}$ is a transverse holomorphic symplectic form on $C_{\operatorname{red}}$. We first claim that $\mathcal{F}_{\operatorname{red}} := (TC_{\operatorname{red}})^{\perp \Omega_{\operatorname{red}}}$ is a vector subbundle of $TC_{\operatorname{red}}$. Fix $x \in C$ and let $y = q(x)$. Here, $q$ denotes the quotient map $C \to C_{\operatorname{red}}$. Let $\mathfrak{g}_{(x)} = \{ \chi_C(a)(x): a \in \mathfrak{g} \}$. Observe that the following is a short exact sequence of real vector spaces:
	\begin{equation}
		\label{Equation 4.6}
		\begin{tikzcd}
			\mathfrak{g}_{(x)} \ar[r] & \mathcal{F}_x \ar[r, "dq_x"] & \mathcal{F}_{\operatorname{red}, y}
		\end{tikzcd}
	\end{equation}
	Since $G$ acts on $C$ freely, $\dim \mathfrak{g}_{(x)} = \dim G$ is independent of $x \in C$. Thus, our claim holds.\par
	Let $T_{\operatorname{quo}}C_{\operatorname{red}} = TC_{\operatorname{red}} / \mathcal{F}_{\operatorname{red}}$. Then $dq: TC \to q^*TC_{\operatorname{red}}$ descends to a vector bundle morphism $\widetilde{dq}: T_{\operatorname{quo}}C \to q^* T_{\operatorname{quo}}C_{\operatorname{red}}$ such that the following diagram is commutative:
	\begin{equation}
		\label{Equation 4.7}
		\begin{tikzcd}
			\mathfrak{g}_C \ar[r] \ar[d] & \mathcal{F} \ar[r, "dq"] \ar[d] & q^*\mathcal{F}_{\operatorname{red}} \ar[d]\\
			\mathfrak{g}_C \ar[r] \ar[d] & TC \ar[r, "dq"] \ar[d] & q^*TC_{\operatorname{red}} \ar[d]\\
			0 \ar[r] & T_{\operatorname{quo}}C \ar[r, "\widetilde{dq}"] & q^*T_{\operatorname{quo}}C_{\operatorname{red}}
		\end{tikzcd}
	\end{equation}
	Here, $\mathfrak{g}_C$ is the image of the vector bundle morphism $C \times \mathfrak{g} \to TC$ given by $(x, a) \mapsto \chi_C(a)(x)$. Since the top two rows and all the columns in the above diagram are short exact sequences of vector bundles, $\widetilde{dq}: T_{\operatorname{quo}}C \to q^* T_{\operatorname{quo}}C_{\operatorname{red}}$ is a vector bundle isomorphism.\par
	As the almost complex structure $I$ on $T_{\operatorname{quo}}C$ is $G$-invariant, it induces an almost complex structure $I_{\operatorname{red}}$ on $T_{\operatorname{quo}}C_{\operatorname{red}}$ such that $\widetilde{dq} \circ I = I_{\operatorname{red}} \circ \widetilde{dq}$. On the other hand, $\Omega_{\operatorname{red}}$ descends to a $2$-form $\widetilde{\Omega}_{\operatorname{red}}$ on $T_{\operatorname{quo}}C_{\operatorname{red}}$. Indeed, for all $x \in C$ and $u, v \in T_{\operatorname{quo}, x}C$,
	\begin{equation*}
		\widetilde{\Omega}_{\operatorname{red}}(\widetilde{dq}(u), \widetilde{dq}(v)) = \widetilde{\Omega}(u, v).
	\end{equation*}
	It forces that for all $u, v \in \Gamma(C_{\operatorname{red}}, T_{\operatorname{quo}}C_{\operatorname{red}})$, $\widetilde{\Omega}_{\operatorname{red}}(I_{\operatorname{red}}(u), v) = \sqrt{-1} \widetilde{\Omega}_{\operatorname{red}}(u, v)$. To conclude, $\Omega_{\operatorname{red}}$ is a transverse holomorphic symplectic form on $C_{\operatorname{red}}$.
\end{proof}

Here come preparations for our proof of the general statement. For $x \in C$, let
\begin{equation*}
	\mathfrak{g}_{(x)} = \{ \chi_C(a)(x): a \in \mathfrak{g} \}
\end{equation*}
and $\widehat{\mathfrak{g}}_{(x)}$ be the preimage of $\{ \widetilde{\chi}_C(a)(x): a \in \mathfrak{g}_\mathbb{C} \}$ under the canonical projection $T_xC \to T_{\operatorname{quo}, x}C$. Clearly, $\mathcal{F}_x \subset \widehat{\mathfrak{g}}_{(x)}$. The following lemma does not require the hypothesis in Theorem \ref{Theorem 4.2}.

\begin{lemma}
	\label{Lemma 4.5}
	Let $x \in C$ and $\mathcal{F}_x^0 = (\ker d\nu_x) \cap (\ker d\nu_x)^{\perp \Omega}$. Then the following statements hold.
	\begin{enumerate}
		\item \label{Lemma 4.5(1)} $\mathcal{F}_x \subset \mathcal{F}_x^0 \subset \ker d\nu_x$, $\ker d\nu_x = \widehat{\mathfrak{g}}_{(x)}^{\perp \Omega}$ and $\mathcal{F}_x^0 = ( \widehat{ \mathfrak{g}}_{(x)} + \ker d\nu_x )^{\perp \Omega}$.
		\item $\ker d\nu_x / \mathcal{F}_x$ and $\mathcal{F}_x^0 / \mathcal{F}_x$ are complex vector subspaces of $T_{\operatorname{quo}, x}C = T_xC / \mathcal{F}_x$.
		\item The rank of $\nu$ at $x$ is $2r_x$, where $r_x$ is the transverse complex rank of the $G$-action at $x$.
	\end{enumerate}
\end{lemma}
\begin{proof}
	\begin{enumerate}
		\item It is clear that $\mathcal{F}_x^0 \subset \ker d\nu_x$ and $\mathcal{F}_x \subset (\ker d\nu_x)^{\perp \Omega}$. Fix $u \in T_xC$. For all $a, b \in \mathfrak{g}$,
		\begin{equation*}
			\langle d\nu(u), a + \sqrt{-1} b \rangle = \Omega(\chi_C(a)(x), u) + \sqrt{-1} \Omega(\chi_C(b)(x), u).
		\end{equation*}
		Thus, $\ker d\nu_x = \mathfrak{g}_{(x)}^{\perp \Omega} \supset \mathcal{F}_x$. It implies that $\mathcal{F}_x \subset \mathcal{F}_x^0$. On the other hand, by Proposition \ref{Proposition A.1}, $\ker d\nu_x = \widehat{\mathfrak{g}}_{(x)}^{\perp \Omega}$. Then $\mathcal{F}_x^0 = \widehat{\mathfrak{g}}_{(x)}^{\perp \Omega} \cap (\ker d\nu_x)^{\perp \Omega} = ( \widehat{ \mathfrak{g}}_{(x)} + \ker d\nu_x )^{\perp \Omega}$.
		\item It directly follows from the above argument and Proposition \ref{Proposition A.1}.
		\item By Proposition \ref{Proposition A.1}, $\dim \ker d\nu_x + \dim \widehat{\mathfrak{g}}_{(x)} = \dim C + \dim \mathcal{F}_x$. Then
		\begin{equation*}
			\dim d\nu_x (T_xC) = \dim \widehat{\mathfrak{g}}_{(x)} - \dim \mathcal{F}_x = 2r_x.
		\end{equation*}
		Note that here $\dim W$ denotes the real dimension of a real vector space $W$.
	\end{enumerate}
\end{proof}

In particular, by Lemma \ref{Lemma 4.5} \eqref{Lemma 4.5(1)}, $d\nu$ descends to a map
\begin{equation*}
	\widetilde{d\nu}: T_{\operatorname{quo}}C \to \mathfrak{g}^* \otimes \mathbb{C},
\end{equation*}
which satisfies the following property: $\widetilde{d\nu}(Iu) = \sqrt{-1} \widetilde{d\nu}(u)$ for all $u \in \Gamma(C, T_{\operatorname{quo}}C)$. Now, let $\mathcal{F}^0 = (TC^0)^{\perp \Omega^0}$ and $\mathcal{E}^0 = \textstyle\bigsqcup_{x \in C^0} \mathcal{E}_x^0$, where
\begin{equation*}
	\mathcal{E}_x^0 := \{ u \in T_xC^0 \otimes \mathbb{C}: \Omega^0(u, v) = 0 \text{ for all } v \in T_xC^0 \otimes \mathbb{C} \}.
\end{equation*}

\begin{lemma}
	\label{Lemma 4.6}
	Under the hypothesis of Theorem \ref{Theorem 4.2}, $(C^0, \Omega^0)$ is a transverse holomorphic symplectic manifold. Moreover,
	\begin{align}
		\mathcal{F}_x^0 = & \widehat{\mathfrak{g}}_{(x)} \quad \text{for all } x \in C^0,\\
		\label{Equation 4.2}
		\mathcal{E}^0 = & \mathcal{F}^0 \otimes \mathbb{C} + \mathcal{E} \cap (TC^0 \otimes \mathbb{C}).
	\end{align}
	In particular, $\Omega^0$ is a basic $2$-form with respect to the inherited $G$-action on $C^0$.
\end{lemma}
\begin{proof}
	We have $d\Omega^0 = (d\Omega) \vert_{C^0} = 0$. We claim that $\mathcal{F}^0$ is a vector bundle. Fix $x \in C^0$. Observe that $\mathcal{F}_x^0 = T_xC^0 \cap (T_xC^0)^{\perp \Omega}$. As $0$ is a clean value of $\nu$, $T_xC^0 = \ker d\nu_x$. Applying Lemma \ref{Lemma 4.5}, we obtain $\mathcal{F}_x^0 = (\widehat{\mathfrak{g}}_{(x)} + T_xC^0)^{\perp \Omega}$. Since $C^0$ is $G$-invariant, $\mathfrak{g}_{(x)} \subset T_xC^0$. According to Lemma \ref{Lemma 4.5} again, $T_xC^0 / \mathcal{F}_x$ is a complex vector subspace of $T_{\operatorname{quo}, x}C$, whence $\widehat{\mathfrak{g}}_{(x)} \subset T_xC^0$. We thus have
	\begin{equation}
		\mathcal{F}_x^0 = (T_xC^0)^{\perp \Omega} = \widehat{\mathfrak{g}}_{(x)}.
	\end{equation}
	By Proposition \ref{Proposition A.1},
	\begin{equation}
		\label{Equation 4.4}
		\dim \mathcal{F}_x^0 + \dim C^0 = \dim C + \dim \mathcal{F}_x.
	\end{equation}
	We can then see from (\ref{Equation 4.4}) that $\dim \mathcal{F}_x^0$ is independent of $x \in C^0$. Our claim holds.\par
	Lemma \ref{Lemma 4.5} states that $\mathcal{F} \vert_{C^0} \subset \mathcal{F}^0 \subset TC^0$, and that $\mathcal{F}^0 / (\mathcal{F} \vert_{C^0})$ and $TC^0 / (\mathcal{F} \vert_{C^0})$ are complex vector subbundles of $(T_{\operatorname{quo}}C, I)$. There is a short exact sequence of real vector bundles over $C^0$:
	\begin{equation}
		\label{Equation 4.5}
		\begin{tikzcd}
			\mathcal{F}^0 / (\mathcal{F} \vert_{C^0}) \ar[r] & TC^0 / (\mathcal{F} \vert_{C^0}) \ar[r] & TC^0 / \mathcal{F}^0 =: T_{\operatorname{quo}}C^0
		\end{tikzcd}
	\end{equation}
	Hence, $I$ induces an almost complex structure $I^0$ on $T_{\operatorname{quo}}C^0$. Note that $\Omega^0$ descends to a $2$-form $\widetilde{\Omega}^0$ on $T_{\operatorname{quo}}C^0$. We can then clearly see that, for all $u, v \in \Gamma(C^0, T_{\operatorname{quo}}C^0)$,
	\begin{equation*}
		\widetilde{\Omega}^0 (I^0 u, v ) = \sqrt{-1} \widetilde{\Omega}^0 ( u, v ).
	\end{equation*}
	In conclusion, $\Omega^0$ is a transverse holomorphic form on $C^0$. Recalling the relationship among $\mathcal{F}, \mathcal{E}, I$ (resp. $\mathcal{F}^0, \mathcal{E}^0, I^0$) in Remark \ref{Remark 2.2}, we can easily deduce the equality (\ref{Equation 4.2}) from (\ref{Equation 4.5}).\par
	As $\Omega$ is $G$-invariant, so is $\Omega^0$. Moreover, for all $x \in C^0$,
	\begin{equation*}
		\mathfrak{g}_{(x)} \subset \widehat{\mathfrak{g}}_{(x)} = \mathcal{F}_x^0 = (T_xC^0)^{\perp \Omega^0}.
	\end{equation*}
	It implies that the $2$-form $\Omega^0$ is basic.
\end{proof}

\begin{proof}[\myproof{Theorem}{\ref{Theorem 4.2}}]
	By hypothesis, $G$ acts on $C^0$ freely. By Lemma \ref{Lemma 4.6}, $\Omega^0$ is a transverse holomorphic symplectic form on $C^0$ and is basic. Hence, applying Lemma \ref{Lemma 4.06} to $(C^0, \Omega^0)$ concludes the proof.
\end{proof}

\begin{remark}
	\label{Remark 4.7}
	Applying (\ref{Equation 4.7}) to $(C^0, \Omega^0)$ shows that we have a short exact sequence of complex vector bundles over $C^0$:
	\begin{center}
		\begin{tikzcd}
			\mathfrak{g}_{C^0} \otimes \mathbb{C} \ar[r] & \mathcal{E}^0 \ar[r, "dq"] & q^*\mathcal{E}_{\operatorname{red}}
		\end{tikzcd}
	\end{center}
	Here, $\mathfrak{g}_{C^0}$ is the image of the vector bundle morphism $C^0 \times \mathfrak{g} \to TC^0$ given by $(x, a) \mapsto \chi_C(a)(x)$.
\end{remark}

\begin{remark}
	As shown in the proof of Theorem \ref{Theorem 4.2}, $\Omega_{\operatorname{red}}$ is a holomorphic symplectic form on $C_{\operatorname{red}}$ if and only if $\mathfrak{g}_{(x)} = \widehat{\mathfrak{g}}_{(x)}$ for all $x \in C^0$.
\end{remark}

\begin{proof}[\myproof{Corollary}{\ref{Corollary 4.3}}]
We know $\dim \nu^{-1}(0) = \dim C - 2r$. Thus, $\dim C_{\operatorname{red}} = \dim C - 2r - \dim G$. By (\ref{Equation 4.4}) and (\ref{Equation 4.6}), $\dim \mathcal{F}_{\operatorname{red}} = \dim \mathcal{F} + 2r - \dim G$. Thus, $\operatorname{rank} \Omega_{\operatorname{red}} = \operatorname{rank} \Omega - 4r$.
\end{proof}

\subsection{Brane reductions, holomorphic symplectic reductions and hyperK\"ahler quotients}
\label{Subsection 4.02}
\quad\par
A compelling application of brane reduction arises when $(C, \Omega)$ is a holomorphic symplectic manifold. In algebraic geometry, one typically considers a reductive group $G_\mathbb{C}$ acting on a smooth holomorphic symplectic variety $(C, \Omega)$ via a Hamiltonian action with moment map $\nu$. Provided a suitable notion of semi-stability exists (or in the affine case), the GIT quotient $\nu^{-1}(0) \sslash G_\mathbb{C}$ defines the \emph{holomorphic symplectic reduction} of $(C, \Omega)$ (see \cite{Kim2016, LeuXieYau2026, Nak1994}).\par
To align with our differential-geometric approach, we compare our results to an analytic analogue of this GIT quotient. Let $C$ be a connected hyperK\"ahler manifold with a triple of complex structures $(I, J, K)$ and associated K\"ahler forms $\omega_I, \omega_J, \omega_K$. Let $G$ act on $(C, \omega_I)$ by a Hamiltonian action with moment map $\mu_I$, and let this action extend to a Hamiltonian $G_\mathbb{C}$-action on $(C, \Omega)$, where $\Omega := \omega_J + \sqrt{-1}\omega_K$, with moment map $\nu = \mu_J + \sqrt{-1} \mu_K: C \to \mathfrak{g}^* \otimes \mathbb{C}$.

\begin{remark}
	By the principal orbit type theorem, if there exists at least one point in $C$ with a trivial stabilizer, then there exists an open dense subset $C^{\operatorname{free}} \subset C$ on which $G$ acts freely. Consequently, $C^{\operatorname{free}} \cap \nu^{-1}(0)$ is an open dense subset of $\nu^{-1}(0)$ where the $G$-action remains free.
\end{remark}

Now, we assume for simplicity that the inherited $G$-action on $C^0 := \nu^{-1}(0)$ is free. Since $(C, I, \omega_I)$ is K\"ahler, for all $x \in C^0$, $\mathfrak{g}_{(x)} \cap I \mathfrak{g}_{(x)} = 0$ and therefore $(\ker d\nu_x)^{\perp \Omega} = \mathfrak{g}_{(x)} \oplus I \mathfrak{g}_{(x)}$, where $\mathfrak{g}_{(x)} := \{ \chi_C(a)(x): a \in \mathfrak{g} \}$. By counting dimensions, we deduce that $0$ is a regular value of $\nu$. Then we obtain a K\"ahler submanifold $C^0$ of $(C, I, \omega_I)$. Applying Theorem \ref{Theorem 4.2}, we obtain the brane reduction $(C_{\operatorname{red}}, \Omega_{\operatorname{red}})$ as a quotient space.\par
This analytic construction is intimately related to the hyperK\"ahler quotient. Note that $G$ acts freely on the level set $\mu_{\operatorname{HK}}^{-1}(0) = \mu_I^{-1}(0) \cap C^0$ , where $\mu_{\operatorname{HK}} := (\mu_I, \mu_J, \mu_K): C \to \mathfrak{g}^* \otimes \mathbb{R}^3$ denotes the hyperK\"ahler moment map. We obtain the hyperK\"ahler quotient $C \tripleslash G := \mu_{\operatorname{HK}}^{-1}(0) / G$, which carries a natural hyperK\"ahler structure $(\omega_{C \tripleslash G, I}, \omega_{C \tripleslash G, J}, \omega_{C \tripleslash G, K})$ \cite{HitKarLinRov1987}.\par
To connect this to the algebraic viewpoint, consider any point $x \in C^0$. Denote by $\overline{G_\mathbb{C} \cdot x}$ the closure of the $G_\mathbb{C}$-orbit $G_\mathbb{C} \cdot x$ in $C^0$. Following Subsection 2.1 of \cite{Sja1995}, we say $x$ is
\begin{itemize}
	\item \emph{semistable} if $\overline{G_\mathbb{C} \cdot x} \cap \mu_{\operatorname{HK}}^{-1}(0) \neq \emptyset$;
	\item \emph{stable} if there exists $y \in \overline{G_\mathbb{C} \cdot x} \cap \mu_{\operatorname{HK}}^{-1}(0)$ such that $d(\mu_I \vert_{C^0})_y: T_yC^0 \to \mathfrak{g}^*$ is surjective.
\end{itemize}
Denote by $(C^0)^{\operatorname{ss}}$ (resp. $(C^0)^{\operatorname{s}}$) the set of semistable (resp. stable) points in $C^0$. Indeed, as $G$ acts on $\mu_{\operatorname{HK}}^{-1}(0)$ freely, $0$ is a regular value of $\mu_I \vert_{C^0}$. It implies that $(C^0)^{\operatorname{ss}} = (C^0)^{\operatorname{s}}$.\par
Introduce a $G$-invariant inner product on $\mathfrak{g}$ so that we obtain the norm-squared of $\mu_I \vert_{C^0}$, i.e. the map $\lVert \mu_I \vert_{C^0} \rVert^2: C^0 \to \mathbb{R}$. We also say the moment map $\mu_I \vert_{C^0}$ is \emph{admissible} if for all point $x \in C^0$, the path of steepest descent $\phi_t(x)$ through $x$ is contained in a compact set, where $\phi_t$ denotes the negative gradient flow of $\lVert \mu_I \vert_{C^0} \rVert^2$. For instance, if $\mu_I \vert_{C^0}$ is proper, then it is admissible. When $\mu_I \vert_{C^0}$ is admissible, we obtain a K\"ahler manifold $(C^0)^{\operatorname{s}} / G_\mathbb{C}$ as a geometric quotient and a diffeomorphism $C \tripleslash G \cong (C^0)^{\operatorname{s}} / G_\mathbb{C}$ \cite{Sja1995}.

\begin{proposition}
	Under the assumption that the $G$-action on $C^0 := \nu^{-1}(0)$ is free, the hyperK\"ahler quotient $C \tripleslash G$ is a submanifold of $C_{\operatorname{red}}$ and
	\begin{equation}
		\label{Equation 4.9}
		\Omega_{\operatorname{red}} \vert_{C \tripleslash G} = \omega_{C \tripleslash G, J} + \sqrt{-1} \omega_{C \tripleslash G, K}.
	\end{equation}
	Assume, in addition, $\mu_I \vert_{C^0}$ is admissible. Then $(C^0)^{\operatorname{s}} / G$ is an open dense subset of $C_{\operatorname{red}}$ containing $C \tripleslash G$, the canonical projection $\pi: (C^0)^{\operatorname{s}} / G \to (C^0)^{\operatorname{s}} / G_\mathbb{C} \cong C \tripleslash G$ exhibits the structure of a real vector bundle of rank $\dim G$, with the inclusion $C \tripleslash G \to (C^0)^{\operatorname{s}} / G$ serving as the zero section, and
	\begin{equation}
		\label{Equation 4.10}
		\Omega_{\operatorname{red}} \vert_{(C^0)^{\operatorname{s}} / G} = \pi^*( \omega_{C \tripleslash G, J} + \sqrt{-1} \omega_{C \tripleslash G, K} ).
	\end{equation}
\end{proposition} 
\begin{proof}
	Since $\mu_{\operatorname{HK}}^{-1}(0)$ is a submanifold of $C^0 = \nu^{-1}(0)$, $C \tripleslash G = \mu_{\operatorname{HK}}^{-1}(0) / G$ is a submanifold of $C_{\operatorname{red}} = C^0/G$. The equality $\Omega \vert_{\mu_{\operatorname{HK}}^{-1}(0)} = \omega_J \vert_{\mu_{\operatorname{HK}}^{-1}(0)} + \sqrt{-1} \omega_K \vert_{\mu_{\operatorname{HK}}^{-1}(0)}$ implies \eqref{Equation 4.9}.\par
	Now, suppose that $\mu_I \vert_{C^0}$ is admissible. By Proposition 2.2 of \cite{Sja1995}, $(C^0)^{\operatorname{s}} = (C^0)^{\operatorname{ss}}$ is a $G_\mathbb{C}$-invariant open subset of $C^0$. Its complement is a complex-analytic subset of $C^0$. Thus, $(C^0)^{\operatorname{s}} / G$ is an open dense subset of $C_{\operatorname{red}}$. Consider the $G$-equivariant smooth map
	\begin{equation}
		\label{Equation 4.11}
		\mathfrak{g} \times \mu_{\operatorname{HK}}^{-1}(0) \to C^0, \quad (a, x) \mapsto \exp(\sqrt{-1}a) \cdot x.
	\end{equation}
	Here, $G$ acts on $\mathfrak{g} \times \mu_{\operatorname{HK}}^{-1}(0)$ by $g \cdot (a, x) := (\operatorname{Ad}(g)(a), g \cdot x)$ and $\operatorname{Ad}$ denotes the adjoint action. By Proposition 1.6 of \cite{Sja1995}, since $G$ acts on $C^0$ freely, so does $G_\mathbb{C}$. By Proposition 2.2 of \cite{Sja1995}, if a $G_\mathbb{C}$-orbit intersects $\mu_{\operatorname{HK}}^{-1}(0)$, then their intersection consists of precisely one $G$-orbit. Then by Cartan decomposition of $G_\mathbb{C}$, i.e. the diffeomorphism $G \times \mathfrak{g} \to G_\mathbb{C}$ given by $(g, a) \mapsto g \exp(\sqrt{-1}a)$, we deduce that \eqref{Equation 4.11} is a smooth embedding.\par
	Proposition 2.2 of \cite{Sja1995} together with Theorem 2.8 of \cite{Sja1995} implies that the image of \eqref{Equation 4.11} is precisely $(C^0)^{\operatorname{s}}$. Taking $G$-quotients, we obtain the associated vector bundle $(\mathfrak{g} \times \mu_{\operatorname{HK}}^{-1}(0))/G$ over $C \tripleslash G$ of rank equal to $\dim G$ and a diffeomorphism
	\begin{equation}
		\label{Equation 4.12}
		(\mathfrak{g} \times \mu_{\operatorname{HK}}^{-1}(0))/G \to (C^0)^{\operatorname{s}} / G
	\end{equation}
	such that the composition of the zero section $C \tripleslash G \to (\mathfrak{g} \times \mu_{\operatorname{HK}}^{-1}(0))/G$ with \eqref{Equation 4.12} is the inclusion map $C \tripleslash G \hookrightarrow (C^0)^{\operatorname{s}} / G$. Under the identifications \eqref{Equation 4.12} and $C \tripleslash G \cong (C^0)^{\operatorname{s}} / G_\mathbb{C}$, the bundle projection $(\mathfrak{g} \times \mu_{\operatorname{HK}}^{-1}(0))/G \to C \tripleslash G$ coincides with the canonical projection $(C^0)^{\operatorname{s}} / G \to (C^0)^{\operatorname{s}} / G_\mathbb{C}$.\par
	Applying Lemma \ref{Lemma 4.5} \eqref{Lemma 4.5(1)}, we deduce that $\Omega \vert_{C^0}$ is a basic $2$-form with respect to the $G_\mathbb{C}$-action. Finally, using \eqref{Equation 4.9}, we can verify that \eqref{Equation 4.10} holds.
\end{proof}

\subsection{Descending $G$-invariant connections onto the quotient space}
\label{Subsection 4.2}
\quad\par
Let $C$ be a smooth $G$-manifold and $(E, \nabla)$ be a $G$-equivariant Hermitian vector bundle over $C$. Denote by $\chi_C$ the induced infinitesimal $\mathfrak{g}$-action on $C$. We say that $\nabla$ \emph{restricts to} the infinitesimal $\mathfrak{g}$-action on $E$ if $\nabla_{\chi_C(a)} = \mathcal{L}_a^E$ for all $a \in \mathfrak{g}$.\par
On the other hand, when $G$ acts on $C$ freely, we say that $(E, \nabla)$ \emph{descends to} a Hermitian vector bundle $(E', \nabla')$ over $C/G$ if $(E, \nabla)$ is isomorphic to, as $G$-equivariant Hermitian vector bundles, the pullback of $(E', \nabla')$ along the quotient map $q: C \to C/G$. Now we give a necessary and sufficient condition for that $(E, \nabla)$ descends to a Hermitian vector bundle $(E', \nabla')$.

\begin{proposition}
	\label{Proposition 4.8}
	Suppose that $G$ acts on $C$ freely. Then $(E, \nabla)$ descends to a Hermitian vector bundle $(E', \nabla')$ over $C/G$ if and only if $\nabla$ restricts to the infinitesimal $\mathfrak{g}$-action on $E$.
\end{proposition}
\begin{proof}
	Assume that $(E, \nabla)$ descends to $(E', \nabla')$. Fix $a \in \mathfrak{g}$. Observe that for any local smooth section $s'$ of $E'$, since $dq(\chi_C(a)) = 0$ and $q^*s'$ is $G$-invariant, $\nabla_{\chi_C(a)} (q^*s') = 0 = \mathcal{L}_a^E (q^*s')$. To show that $\nabla_{\chi_C(a)} = \mathcal{L}_a^E$, it remains to see that we can cover $C$ by open subsets over which there is a local frame of $E$ of the form $(q^*s'_1, ..., q^*s'_r)$, where $(s'_1, ..., s'_r)$ is a local frame of $E'$.\par
	Conversely, assume that $\nabla$ restricts to the infinitesimal $\mathfrak{g}$-action on $E$. It is clear that $E$ is isomorphic to $q^*E'$ as $G$-equivariant Hermitian vector bundles over $C$, where $E' = E / G$. Consider any $s' \in \Gamma(C/G, E')$, $y \in C/G$ and $v \in T_y(C/G)$. Define
	\begin{equation*}
		\nabla'_v s' := G \cdot (\nabla_u s) \in E'_y,
	\end{equation*}
	where $s = q^*s'$ and $u \in T_xC$ for some $x \in q^{-1}(y)$ such that $dq (u) = v$. We claim that this is well defined. Suppose $x' \in q^{-1}(y)$, $u' \in T_{x'}C$ and $dq(u') = v$. Pick $g \in G$ such that $x' = \rho_C(g)(x)$, where $\rho_C$ is the $G$-action on $C$. Then $dq(u' - d\rho_C(g)(u)) = 0$. There exists $a \in \mathfrak{g}$ such that $u' - d\rho_C(g)(u) = \chi_C(a)_{x'}$. As $s$ is $G$-invariant and $\nabla$ restricts to the infinitesimal $\mathfrak{g}$-action on $E$,
	\begin{equation*}
		g \cdot (\nabla_u s) = \nabla_{d\rho_C(g)(u)} s = \nabla_{u'}s + (\mathcal{L}_a^E s)(x') = \nabla_{u'}s.
	\end{equation*}
	Eventually, we can easily verify that $\nabla'$ is a unitary connection on $E'$.\par
\end{proof}

\subsection{Reduction of $G$-invariant coisotropic A-branes}
\label{Subsection 4.3}
\quad\par
Consider a $G$-invariant coisotropic A-brane $\mathcal{B} = (C, E, \nabla)$ on $(X, \omega)$ with moment section $\mu_\mathcal{B} \in \Gamma(C, \mathfrak{g}^* \otimes \operatorname{End}(E))$ and moment trace $\nu: C \to \mathfrak{g}^* \otimes \mathbb{C}$. We will see how to obtain a coisotropic A-brane on the symplectic reduction $(X \sslash G, \omega_{\operatorname{red}})$ from $\mathcal{B}$ based on an assumption on $\mu_\mathcal{B}$.

\begin{definition}
	$\mu_\mathcal{B}$ is said to be \emph{neat} if
	\begin{enumerate}
		\item $0$ is a clean value of $\nu$.
		\item $\nabla \vert_{\nu^{-1}(0)}$ restricts to the infinitesimal $\mathfrak{g}$-action on $E \vert_{\nu^{-1}(0)}$.
	\end{enumerate}
\end{definition}

\begin{remark}
	If $E$ is of rank $1$, then the second condition in the above definition is superfluous.
\end{remark}

As $G$ acts on $\mu^{-1}(0)$ freely and $C^0 := \nu^{-1}(0) \subset \mu^{-1}(0)$, $G$ also acts on $C^0$ freely. Under the first condition of the above definition, we can apply Theorem \ref{Theorem 4.2} to obtain the brane reduction $(C_{\operatorname{red}}, \Omega_{\operatorname{red}})$ of $(C, \Omega)$, where $\Omega = F_{\operatorname{tr}}^\nabla + \sqrt{-1} \omega \vert_C$; by Proposition \ref{Proposition 4.8}, the second condition implies that $(E^0, \nabla^0) := (E \vert_{C^0}, \nabla \vert_{C^0})$ descends to a Hermitian vector bundle $(E_{\operatorname{red}}, \nabla_{\operatorname{red}})$ over $C_{\operatorname{red}}$. Now, we state our main theorem.

\begin{theorem}[$=$ Theorem \ref{Theorem 1.2}]
	\label{Theorem 4.11}
	Let $\mathcal{B} = (C, E, \nabla)$ be a $G$-invariant A-brane on $(X, \omega)$. If the moment section $\mu_\mathcal{B}$ of $\mathcal{B}$ is neat, then $\mathcal{B}_{\operatorname{red}} = (C_{\operatorname{red}}, E_{\operatorname{red}}, \nabla_{\operatorname{red}})$ is an A-brane on $(X \sslash G, \omega_{\operatorname{red}})$.
\end{theorem}

The following lemma does not require $0$ to be a clean value of $\nu$. In the sequel, we keep using notations in Subsection \ref{Subsection 4.1}.

\begin{lemma}
	\label{Lemma 4.12}
	Let $x \in C$. Then the following statements hold.
	\begin{enumerate}
		\item $(\ker d\nu_x)^{\perp \omega} = (\ker d\nu_x)^{\perp \Omega} \subset T_xC$.
		\item $(\ker d\mu_x) \cap (\ker d\nu_x)^{\perp \omega} = ( \mathfrak{g}_{(x)} + \ker d\nu_x )^{\perp \omega}$.
		\item $(\ker d\nu_x) \cap (\ker d\nu_x)^{\perp \Omega} = ( \widehat{\mathfrak{g}}_{(x)} + \ker d\nu_x)^{\perp \omega}$.
	\end{enumerate}
\end{lemma}
\begin{proof}
	\begin{enumerate}
		\item As $C$ is a coisotropic submanifold of $X$, $(T_xC)^{\perp \omega} \subset T_xC$. Then $\mathcal{F}_x = (T_xC)^{\perp \omega}$. By Lemma \ref{Lemma 4.5}, $\mathcal{F}_x \subset \ker d\nu_x$ and $\ker d\nu_x / \mathcal{F}_x$ is a complex vector subspace of $(T_{\operatorname{quo}, x}C, I_x)$. Thus, $(\ker d\nu_x)^{\perp \omega} \subset \mathcal{F}_x^{\perp \omega} = T_xC$. By Proposition \ref{Proposition A.1},
		\begin{equation*}
			(\ker d\nu_x)^{\perp \omega} = T_xC \cap (\ker d\nu_x)^{\perp \omega} = (\ker d\nu_x)^{\perp \Omega}.
		\end{equation*}
		\item Since $\ker d\mu_x = \mathfrak{g}_{(x)}^{\perp \omega}$, $(\ker d\mu_x) \cap \left( \ker d\nu_x \right)^{\perp \omega} = ( \mathfrak{g}_{(x)} + \ker d\nu_x )^{\perp \omega}$.
		\item By Lemma \ref{Lemma 4.5}, $(\ker d\nu_x) \cap (\ker d\nu_x)^{\perp \Omega} = ( \widehat{\mathfrak{g}}_{(x)} + \ker d\nu_x)^{\perp \Omega}$. Now, we have
		\begin{equation*}
			( \widehat{\mathfrak{g}}_{(x)} + \ker d\nu_x)^{\perp \omega} \subset (\ker d\nu_x)^{\perp \omega} \subset T_xC.
		\end{equation*}
		Then by Proposition \ref{Proposition A.1}, $( \widehat{\mathfrak{g}}_{(x)} + \ker d\nu_x)^{\perp \Omega} = T_xC \cap ( \widehat{\mathfrak{g}}_{(x)} + \ker d\nu_x)^{\perp \omega} = ( \widehat{\mathfrak{g}}_{(x)} + \ker d\nu_x)^{\perp \omega}$.
	\end{enumerate}
\end{proof}

\begin{proof}[\myproof{Theorem}{\ref{Theorem 1.2}}]
	Note that $\Omega_{\operatorname{red}} = F_{\operatorname{tr}}^{\nabla_{\operatorname{red}}} + \sqrt{-1} \omega_{\operatorname{red}} \vert_{C_{\operatorname{red}}}$.  
	Now, we want to show that $C_{\operatorname{red}}$ is a coisotropic submanifold of $(X \sslash G, \omega_{\operatorname{red}})$, i.e. that the symplectic orthogonal complement $(TC_{\operatorname{red}})^{\perp \omega_{\operatorname{red}}}$ in $T(X \sslash G) \vert_C$ is a vector subbundle of $TC_{\operatorname{red}}$. It suffices to prove the claim that
	\begin{equation*}
		(TC_{\operatorname{red}})^{\perp \omega_{\operatorname{red}}} = (TC_{\operatorname{red}})^{\perp \Omega_{\operatorname{red}}},
	\end{equation*}
	because, by Theorem \ref{Theorem 4.2}, $\Omega_{\operatorname{red}}$ is a transverse holomorphic symplectic form on $C_{\operatorname{red}}$ and hence $(TC_{\operatorname{red}})^{\perp \Omega_{\operatorname{red}}}$ is a vector subbundle of $TC_{\operatorname{red}}$.\par
	Fix $x \in C^0$. By Lemma \ref{Lemma 4.12},
	\begin{itemize}
		\item $(T_xC^0)^{\perp \omega^0} = T_x\mu^{-1}(0) \cap \left( T_xC^0 \right)^{\perp \omega} = (\mathfrak{g}_{(x)} + T_xC^0)^{\perp \omega}$, where $\omega^0 = \omega \vert_{\mu^{-1}(0)}$.
		\item $(T_xC^0)^{\perp \Omega^0} = T_xC^0 \cap (T_xC^0)^{\perp \Omega} = (\widehat{\mathfrak{g}}_{(x)} + T_xC^0)^{\perp \omega}$.
	\end{itemize}
	As $\mathfrak{g}_{(x)}$ and $\widehat{\mathfrak{g}}_{(x)}$ are vector subspaces of $T_xC^0$, $(T_xC^0)^{\perp \omega^0} = (T_xC^0)^{\perp \Omega^0}$. Let $q: C^0 \to C_{\operatorname{red}}$ be the quotient map. The short exact sequences
	\begin{align*}
		\mathfrak{g}_{(x)} \longrightarrow (T_xC^0)^{\perp \omega^0} \overset{dq_x}{\longrightarrow} (T_{q(x)} C_{\operatorname{red}})^{\perp \omega_{\operatorname{red}}}\\
		\mathfrak{g}_{(x)} \longrightarrow (T_xC^0)^{\perp \Omega^0} \overset{dq_x}{\longrightarrow} (T_{q(x)} C_{\operatorname{red}})^{\perp \Omega_{\operatorname{red}}}
	\end{align*}
	for all $x \in C^0$ imply that our claim holds.\par
	It remains to show that $F_0^{\nabla_{\operatorname{red}}}$ is of transverse type $(1, 1)$. Fix $x \in C^0$. Note that $(E^0, \nabla^0)$ is still $G$-equivariant. Since $\nabla^0$ restricts to the infinitesimal $\mathfrak{g}$-action on $E^0$, by Proposition \ref{Proposition 3.2},
	\begin{equation}
		\label{Equation 4.8}
		F_0^{\nabla^0}(u, v) = 0,
	\end{equation}
	for all $u \in \mathfrak{g}_{(x)} + \mathcal{F}_x$ and $v \in T_xC^0$. Since $F_0^\nabla$ is of transverse type $(1, 1)$, (\ref{Equation 4.8}) holds even when $u \in \widehat{\mathfrak{g}}_{(x)} = \mathcal{F}_x^0$. By simple linear algebra, we can show that $\widetilde{F}_0^{\nabla^0}(I^0u, I^0v) = \widetilde{F}_0^{\nabla^0}(u, v)$ for all $u, v \in T_{\operatorname{quo}, x}C^0$. It implies that $F_0^{\nabla_{\operatorname{red}}}$ is of transverse type $(1, 1)$. We are done.
\end{proof}

\section{A-model morphism spaces and sheaf cohomologies on intersections of A-branes}
\label{Section 5}
Let $\mathcal{B} = (C, E, \nabla)$, $\mathcal{B}' = (C', E', \nabla')$ be A-branes on $(X, \omega)$. Throughout this section, we assume that $S := C \cap C'$ is a clean intersection. Consider the following Hermitian vector bundle over $S$:
\begin{equation}
	\underline{\operatorname{Hom}}(\mathcal{B}, \mathcal{B}') := \underline{\operatorname{Hom}}(E \vert_S, E' \vert_S),
\end{equation}
with the induced unitary connection $\nabla^{\underline{\operatorname{Hom}}(E \vert_S, E' \vert_S)}$. When $S$ admits a transverse holomorphic structure $\mathcal{E}_S$ such that $\nabla^{\underline{\operatorname{Hom}}(E \vert_S, E' \vert_S)}$ restricts to a flat $\mathcal{E}_S$-connection, we can take the sheaf cohomology of transversely holomorphic sections of $\underline{\operatorname{Hom}}(\mathcal{B}, \mathcal{B}')$:
\begin{equation}
	\label{Equation 5.2}
	H_{\mathcal{E}_S}^*(S, \underline{\operatorname{Hom}}(\mathcal{B}, \mathcal{B}')).
\end{equation}
We will see from Subsection \ref{Subsection 5.1} that, for suitably chosen $\mathcal{E}_S$, (\ref{Equation 5.2}) plays an important role in mathematically formulating the A-model morphism space $\operatorname{Hom}_X(\mathcal{B}, \mathcal{B}')$ from $\mathcal{B}$ to $\mathcal{B}'$, at least in the classical limit.\par
Indeed, a choice of $\mathcal{E}_S$ should be constrained by the intersection of the associated transverse holomorphic structures $\mathcal{E}$ on $C$ and $\mathcal{E}'$ on $C'$. Due to this issue, we will introduce the notion of \emph{adaptedness} of a transverse holomorphic structure on $S$ in Subsection \ref{Subsection 5.2}.

\subsection{A case study on A-model morphism spaces}
\label{Subsection 5.1}
\quad\par
This subsection is a brief review on the physical background behind mathematical developments of morphisms spaces among coisotropic A-branes.\par
When $\mathcal{B}, \mathcal{B}'$ are Lagrangian, $\operatorname{Hom}_X(\mathcal{B}, \mathcal{B}')$ is mathematically formulated as the Floer cohomology for $(\mathcal{B}, \mathcal{B}')$ \cite{Fuk1993}. As $\mathcal{E} = TC \otimes \mathbb{C}$ and $\mathcal{E}' = TC' \otimes \mathbb{C}$, $\mathcal{E} \cap \mathcal{E}' = TS \otimes \mathbb{C}$ is a transverse holomorphic structure on $S$. Then $H_{\mathcal{E} \cap \mathcal{E}'}^*(S, \underline{\operatorname{Hom}}(\mathcal{B}, \mathcal{B}'))$, which is the sheaf cohomology of flat sections of $\underline{\operatorname{Hom}}(\mathcal{B}, \mathcal{B}')$, can be viewed as a classical limit of the above Floer cohomology (see e.g. \cite{FukOhOhtOno2009}).\par
Next, consider two identical rank-$1$ A-branes $\mathcal{B}, \mathcal{B}'$ on $(X, \omega)$, i.e. $\mathcal{B} = \mathcal{B}'$. Then $\mathcal{E} \cap \mathcal{E}' = \mathcal{E}$ is surely a transverse holomorphic structure on $S = C$ and $\underline{\operatorname{Hom}}(\mathcal{B}, \mathcal{B})$ is trivial. The classical limit of $\operatorname{Hom}_X(\mathcal{B}, \mathcal{B})$ is regarded as the BRST cohomology of open strings with both ends on $\mathcal{B}$ \cite{KapLi2005, KapOrl2003}, and coincides with $H_{\mathcal{E}}^*(S, \underline{\operatorname{Hom}}(\mathcal{B}, \mathcal{B}))$.\par
Finally, consider the situation when $\mathcal{B} = (M, M \times \mathbb{C}, d)$ for some Lagrangian submanifold $M$ of $(X, \omega)$ and $\mathcal{B}' = \mathcal{B}_{\operatorname{cc}} = (X, L, \nabla)$ is a rank-$1$ space-filling A-brane. Denote by $\Omega$ the holomorphic symplectic form on $X$ associated with $\mathcal{B}_{\operatorname{cc}}$. Note that $S = M$, $\mathcal{E} = TM \otimes \mathbb{C}$, $\mathcal{E}' = T^{0, 1}X$ and $\underline{\operatorname{Hom}}(\mathcal{B}, \mathcal{B}_{\operatorname{cc}}) = (L \vert_M, \nabla \vert_M)$. There are two extreme cases for $\mathcal{E} \cap \mathcal{E}' = (TM \otimes \mathbb{C}) \cap T^{0, 1}X$.\par
\begin{itemize}
	\item Suppose that $\mathcal{E} \cap \mathcal{E}'$ defines a complex structure on $M$, or equivalently, that $M$ is a complex Lagrangian submanifold of $(X, \Omega)$. In this case, $\nabla \vert_M$ is flat, giving $L \vert_M$ a holomorphic structure. Physical arguments \cite{GaiWit2022, GukWit2009, KapWit2007} show that we should regard $\operatorname{Hom}_X(\mathcal{B}, \mathcal{B}_{\operatorname{cc}})$ as $H_{\overline{\partial}}^{0, *}(M, L \vert_M \otimes \sqrt{K_M})$, assuming that $M$ has a square root $\sqrt{K_M}$ of its canonical bundle. It only differ from $H_{\mathcal{E}_S}^*(S, \underline{\operatorname{Hom}}(\mathcal{B}, \mathcal{B}_{\operatorname{cc}}))$ by the factor $\sqrt{K_M}$ for $\mathcal{E}_S = \mathcal{E} \cap \mathcal{E}' = T^{0, 1}M$.
	\item Suppose that $\mathcal{E} \cap \mathcal{E}'$ is a zero vector bundle, or equivalently, that $M$ is a symplectic submanifold of $(X, \operatorname{Re} \Omega)$. It is not known how to formulate $\operatorname{Hom}_X(\mathcal{B}, \mathcal{B}_{\operatorname{cc}})$ mathematically in general. However, physical arguments \cite{GaiWit2022, GukWit2009, KapWit2007} show that, under certain assumptions which imply that $(M, \operatorname{Re} \Omega \vert_M)$ is compact K\"ahler (and $\sqrt{K_M}$ exists), $\operatorname{Hom}_X(\mathcal{B}, \mathcal{B}_{\operatorname{cc}})$ should be regarded as $H_{\overline{\partial}}^{0, *}(M, L \vert_M \otimes \sqrt{K_M})$. Again, it only differs from $H_{\mathcal{E}_S}^*(S, \underline{\operatorname{Hom}}(\mathcal{B}, \mathcal{B}_{\operatorname{cc}}))$ by the factor $\sqrt{K_M}$, if we take $\mathcal{E}_S$ to be the $\operatorname{Re} \Omega \vert_M$-compatible complex structure on $S = M$. The key difference from the previous case is that $\mathcal{E}_S$ in this case does not come from the intersection $\mathcal{E} \cap \mathcal{E}'$, but extra assumptions on $M$ and $X$.
\end{itemize}
All the above cases show that A-model morphism spaces are closely related to the behaviour of the intersection $\mathcal{E} \cap \mathcal{E}'$ -- whether $\mathcal{E} \cap \mathcal{E}'$ defines a transverse holomorphic structure on $S$ or not. This issue is the motivation of the next subsection.

\subsection{Adapted transverse holomorphic structures on intersections of A-branes}
\label{Subsection 5.2}
\quad\par
We first state basic properties of $\mathcal{E} \cap \mathcal{E}'$ (and $\mathcal{F} \cap \mathcal{F}'$, where $\mathcal{F} = (TC)^{\perp \omega}$ and $\mathcal{F}' = (TC')^{\perp \omega}$). Recall that the \emph{excess bundle} of $S$ is the vector bundle
\begin{equation*}
	\frac{TX \vert_S}{TC \vert_S + TC' \vert_S}.
\end{equation*}
Its rank, denoted by $e(C, C')$, is given by $e(C, C') = \dim X + \dim S - \dim C - \dim C'$.

\begin{proposition}
	\label{Proposition 5.1}
	$\mathcal{F} \cap \mathcal{F}'$ is an involutive rank-$e(C, C')$ vector subbundle of $TS$, and for all $x \in S$,
	\begin{enumerate}
		\item the complex dimension of $\mathcal{E}_x \cap \mathcal{E}'_x$ is at most $\tfrac{1}{2}(\dim S + e(C, C'))$; and
		\item for all $u, v \in \mathcal{E}_x \cap \mathcal{E}'_x$, $F^{\nabla^{\underline{\operatorname{Hom}}(E \vert_S, E' \vert_S)}}(u, v) = 0$.
	\end{enumerate}
\end{proposition}
\begin{proof}
	Observe that $\mathcal{F}_x \cap \mathcal{F}'_x = (T_xC)^{\perp \omega} \cap (T_xC')^{\perp \omega} = (T_xC + T_xC')^{\perp \omega}$. We have
	\begin{equation*}
		\dim \mathcal{F}_x \cap \mathcal{F}'_x = \dim X - \dim (T_xC + T_xC') = e(C, C').
	\end{equation*}
	Thus, $\mathcal{F} \cap \mathcal{F}'$ is a vector subbundle of $TS$. As $\mathcal{F}, \mathcal{F}'$ are involutive, so is $\mathcal{F} \cap \mathcal{F}'$.\par
	As $\mathcal{E}_x \cap \mathcal{E}'_x \cap \overline{\mathcal{E}_x \cap \mathcal{E}'_x} = (\mathcal{F}_x \cap \mathcal{F}'_x) \otimes \mathbb{C} \subset T_xS \otimes \mathbb{C}$, the complex dimension of $\mathcal{E}_x \cap \mathcal{E}'_x$ is at most
	\begin{equation*}
		\dim (\mathcal{F}_x \cap \mathcal{F}'_x) + \tfrac{1}{2}(\dim S - \dim (\mathcal{F}_x \cap \mathcal{F}'_x)) = \tfrac{1}{2} (\dim S + e(C, C')).
	\end{equation*}
	Fix $u, v \in \mathcal{E}_x \cap \mathcal{E}'_x$. Then for any $A \in \operatorname{Hom}(E_x, E'_x)$,
	\begin{align*}
		F^{\nabla^{\underline{\operatorname{Hom}}(E \vert_S, E' \vert_S)}}(u, v)(A) = & F^{\nabla'}(u, v) \circ A - A \circ F^\nabla(u, v)\\
		= & (F^{\nabla'}(u, v) + \omega(u, v) \operatorname{Id}_{E'_x}) \circ A - A \circ (F^\nabla(u, v) + \omega(u, v) \operatorname{Id}_{E_x}).
	\end{align*}
	The last line vanishes due to Remark \ref{Remark 2.5} for the two A-branes $\mathcal{B}$ and $\mathcal{B}'$.
\end{proof}

Motivated by the above proposition, we are interested in the following choice of $\mathcal{E}_S$.

\begin{definition}
	A transverse holomorphic structure $\mathcal{E}_S$ on $S$ is said to be \emph{adapted} if \begin{enumerate}
		\item $\mathcal{E} \cap \mathcal{E}' \subset \mathcal{E}_S$ and $\mathcal{F} \cap \mathcal{F}' = \mathcal{E}_S \cap \overline{\mathcal{E}_S} \cap TS$; and
		\item the curvature $F^{\nabla^{\underline{\operatorname{Hom}}(E \vert_S, E' \vert_S)}}$ vanishes on $\mathcal{E}_S$.
	\end{enumerate}
\end{definition}

\begin{remark}
	\label{Remark 5.3}
	If $\mathcal{E} \cap \mathcal{E}'$ itself is a transverse holomorphic structure on $S$, then it is the unique adapted transverse holomorphic structure on $S$. In particular, the intersection of two identical A-branes always has a unique adapted transverse holomorphic structure.
\end{remark}

For an adapted transverse holomorphic structure $\mathcal{E}_S$ on $S$, by definition, $\nabla^{\underline{\operatorname{Hom}}(E \vert_S, E' \vert_S)}$ restricts to a flat $\mathcal{E}_S$-connection $d_{\mathcal{E}_S}^{\underline{\operatorname{Hom}}(E \vert_S, E' \vert_S)}$ so that we can obtain $H_{\mathcal{E}_S}^*(S, \underline{\operatorname{Hom}}(\mathcal{B}, \mathcal{B}'))$, which is isomorphic to the cohomology of the Chevalley-Eilenberg complex
\begin{equation*}
	\left( \Gamma\left(S, \bigwedge \mathcal{E}_S^* \otimes \underline{\operatorname{Hom}}(E \vert_S, E' \vert_S) \right), d_{\mathcal{E}_S}^{\underline{\operatorname{Hom}}(E \vert_S, E' \vert_S)} \right).
\end{equation*}
When $S$ is compact, $H_{\mathcal{E}_S}^*(S, \underline{\operatorname{Hom}}(\mathcal{B}, \mathcal{B}'))$ is finite dimensional. We suspect that $H_{\mathcal{E}_S}^*(S, \underline{\operatorname{Hom}}(\mathcal{B}, \mathcal{B}'))$ should be a mathematical realization of (the classical limit of) $\operatorname{Hom}_X(\mathcal{B}, \mathcal{B}')$, up to a correction by some $\operatorname{Spin}^{\operatorname{c}}$-structures on the supports of $\mathcal{B}$ and $\mathcal{B}'$.

\section{Intersections of A-branes commute with brane reduction}
\label{Section 6}
\quad\par
Let $\mathcal{B}, \mathcal{B}'$ be $G$-invariant A-branes on the Hamiltonian $G$-manifold $(X, \omega)$ with neat moment sections so that we can obtain their brane reductions $\mathcal{B}_{\operatorname{red}}, \mathcal{B}'_{\operatorname{red}}$. Again, assume that $S := C \cap C'$ is a clean intersection, where $C, C'$ are the supports of $\mathcal{B}, \mathcal{B}'$ respectively. We will examine the relationship between $S$ and $S_{\operatorname{red}} := C_{\operatorname{red}} \cap (C')_{\operatorname{red}}$ (resp. $\underline{\operatorname{Hom}}(\mathcal{B}, \mathcal{B}')$ and $\underline{\operatorname{Hom}}(\mathcal{B}_{\operatorname{red}}, \mathcal{B}'_{\operatorname{red}})$). Let $\nu: C \to \mathfrak{g}^* \otimes \mathbb{C}$, $\nu': C' \to \mathfrak{g}^* \otimes \mathbb{C}$ be the moment traces of $\mathcal{B}$ and $\mathcal{B}'$ respectively (see Proposition \ref{Proposition 3.6}). Recall that $\mu: X \to \mathfrak{g}^*$ denotes the given moment map of $(X, \omega)$. Define a $G$-equivariant smooth map
\begin{equation}
	\label{Equation 6.1}
	\eta := (\operatorname{Re} \nu \vert_S, \operatorname{Re} \nu' \vert_S, \mu \vert_S): S \to \mathfrak{g}^* \otimes \mathbb{R}^3.
\end{equation}
We can clearly see that $S^0 := \eta^{-1}(0) = \nu^{-1}(0) \cap (\nu')^{-1}(0)$.

\begin{proposition}
	\label{Proposition 6.1}
	The following conditions are equivalent.
	\begin{enumerate}
		\item \label{Proposition 6.1 (1)}
		$0 \in \mathfrak{g}^* \otimes \mathbb{R}^3$ is a clean value of $\eta$.
		\item \label{Proposition 6.1 (2)}
		$\nu^{-1}(0)$ and $(\nu')^{-1}(0)$ intersect cleanly in $\mu^{-1}(0)$.
		\item \label{Proposition 6.1 (3)}
		$C_{\operatorname{red}}$ and $C'_{\operatorname{red}}$ intersect cleanly in $X \sslash G$.
	\end{enumerate}
	If any one of the above equivalent conditions holds, then $S^0$ is a principal $G$-bundle over $S_{\operatorname{red}}$ and the Hermitian vector bundle $\underline{\operatorname{Hom}}(\mathcal{B}, \mathcal{B}') \vert_{S^0}$ over $S^0$ descends to the Hermitian vector bundle $\underline{\operatorname{Hom}}(\mathcal{B}_{\operatorname{red}}, \mathcal{B}'_{\operatorname{red}})$ over $S_{\operatorname{red}}$.
\end{proposition}
\begin{proof}
	Let $C^0 = \nu^{-1}(0)$ and $(C')^0 = (\nu')^{-1}(0)$. Fix $x \in \eta^{-1}(0)$. Observe that
	\begin{align*}
		T_xC^0 = \{ u \in T_xC: d\nu(u) = 0 \} \quad \text{and} \quad T_x(C')^0 = \{ u \in T_xC': d\nu'(u) = 0 \}.
	\end{align*}
	By the hypothesis that $T_xS = T_xC \cap T_xC'$,
	\begin{equation*}
		T_xC^0 \cap T_x(C')^0 = \{ u \in T_xS: d\nu(u) = d\nu'(u) = 0 \} = \ker d\eta_x.
	\end{equation*}
	It implies that Conditions (\ref{Proposition 6.1 (1)}) and (\ref{Proposition 6.1 (2)}) are equivalent.\par
	Note that $\mu^{-1}(0)$ is a principal $G$-bundle over $X \sslash G$. Let $q: \mu^{-1}(0) \to X \sslash G$ be the bundle projection. Observe that $C^0 = q^{-1}(C_{\operatorname{red}})$ and $(C')^0 = q^{-1}(C'_{\operatorname{red}})$. Hence, $\eta^{-1}(0) = q^{-1}(S_{\operatorname{red}})$. We can then easily show that Conditions (\ref{Proposition 6.1 (2)}) and (\ref{Proposition 6.1 (3)}) are equivalent.\par
	It is then evident that the conclusion is true when one of the equivalent conditions holds.
\end{proof}
Therefore, further assuming that the submanifolds $C_{\operatorname{red}}, C'_{\operatorname{red}}$ of $X \sslash G$ intersect cleanly, we can obtain $S_{\operatorname{red}}$ from $S$ (resp. $\underline{\operatorname{Hom}}(\mathcal{B}_{\operatorname{red}}, \mathcal{B}'_{\operatorname{red}})$ from $\underline{\operatorname{Hom}}(\mathcal{B}, \mathcal{B}')$) by the following procedure:
\begin{enumerate}
	\item restricting it to the submanifold $\eta^{-1}(0)$, and then
	\item taking the $G$-quotient of this restriction.
\end{enumerate}
Loosely speaking, we conclude that \emph{intersections of A-branes commute with brane reduction}.

\begin{example}
	Endow $X = \{ (z^1, z^2, w^1, w^2) \in \mathbb{C}^4: (z^1, z^2) \neq (0, 0) \}$ with the symplectic form $\omega = \tfrac{\sqrt{-1}}{2} \sum_{i=1}^2 (dz^i \wedge d\overline{z}^i + dw^i \wedge d\overline{w}^i)$. Then $G = \mathbb{S}^1$ acts on $(X, \omega)$ via a free Hamiltonian action
	\begin{equation*}
		\lambda \cdot (z^1, z^2, w^1, w^2) = (\lambda z^1, \lambda z^2, \lambda^{-1} w^1, \lambda^{-1} w^2),
	\end{equation*}
	with moment map $\langle \mu, a \rangle = -\tfrac{1}{2} ( \lvert z^1 \rvert^2 + \lvert z^2 \rvert^2 - \lvert w^1 \rvert^2 - \lvert w^1 \rvert^2 - 1 )$, where $a \in \mathfrak{g}$ is the generator whose fundamental vector field is given by
	\begin{equation*}
		\chi(a) = \sqrt{-1} \textstyle\sum_{i=1}^2 \left( z^i \tfrac{\partial}{\partial z^i} - \overline{z}^i \tfrac{\partial}{\partial \overline{z}^i} - w^i \tfrac{\partial}{\partial w^i} + \overline{w}^i \tfrac{\partial}{\partial \overline{w}^i} \right).
	\end{equation*}
	Consider the trivial Hermitian line bundle $X \times \mathbb{C}$ equipped with the $\mathbb{S}^1$-action given by
	\begin{equation*}
		\lambda \cdot (z^1, z^2, w^1, w^2, c) = (\lambda z^1, \lambda z^2, \lambda^{-1} w^1, \lambda^{-1} w^2, \lambda c),
	\end{equation*}
	covering the $\mathbb{S}^1$-action on $X$. We then obtain two $\mathbb{S}^1$-invariant space-filling A-branes on $(X, \omega)$:
	\begin{align*}
		\mathcal{B} := & \left(X, X \times \mathbb{C}, d - \tfrac{\sqrt{-1}}{2} \textstyle\sum_{i=1}^2 \operatorname{Re} \left( z^i dw^i - w^i dz^i \right) \right),\\
		\mathcal{B}' := & \left(X, X \times \mathbb{C}, d - \tfrac{\sqrt{-1}}{2} \textstyle\sum_{i=1}^2 \operatorname{Im} \left( z^i dw^i - w^i dz^i \right) \right).
	\end{align*}
	The respective holomorphic symplectic forms on $X$ associated with $\mathcal{B}, \mathcal{B}'$ are of the form $F + \sqrt{-1} \omega$ and $F' + \sqrt{-1} \omega$, where
	\begin{equation*}
		F := \operatorname{Re} (dz^1 \wedge dw^1 + dz^2 \wedge dw^2) \quad \text{and} \quad F' := \operatorname{Im} (dz^1 \wedge dw^1 + dz^2 \wedge dw^2).
	\end{equation*}
	The real parts of the respective moment traces $\nu, \nu': X \to \mathfrak{g}^* \otimes \mathbb{C}$ of $\mathcal{B}, \mathcal{B}'$ are
	\begin{align*}
		\langle \operatorname{Re} \nu, a \rangle (z^1, z^2, w^1, w^2) = & \operatorname{Re} (\sqrt{-1} (z^1 w^1 + z^2 w^2));\\
		\langle \operatorname{Re} \nu', a \rangle (z^1, z^2, w^1, w^2) = & \operatorname{Im} (\sqrt{-1} (z^1 w^1 + z^2 w^2)).
	\end{align*}
	Indeed, $X$ is a hyperK\"ahler manifold with the triple of K\"ahler forms $(F, F', \omega)$. The $\mathbb{S}^1$-action is Hamiltonian with respect to all three forms, with the hyperK\"ahler moment map given by:
	\begin{equation*}
		\eta := (\operatorname{Re} \nu, \operatorname{Re} \nu', \mu): X \to \mathfrak{g}^* \otimes \mathbb{R}^3.
	\end{equation*}
	The intersection of the supports of the two brane reductions, $\mathcal{B}_{\operatorname{red}}$ and $\mathcal{B}'_{\operatorname{red}}$, is precisely the hyperK\"ahler quotient $X \tripleslash \mathbb{S}^1 := \eta^{-1}(0) / \mathbb{S}^1$.
\end{example}

It is intriguing to study the relationship between the sheaf cohomologies of transversely holomorphic sections of $\underline{\operatorname{Hom}}(\mathcal{B}, \mathcal{B}')$ and $\underline{\operatorname{Hom}}(\mathcal{B}_{\operatorname{red}}, \mathcal{B}'_{\operatorname{red}})$ when both $S$ and $S_{\operatorname{red}}$ admit adapted transversely holomorphic structures. In the next section, we will examine these sheaf cohomologies in some special cases, assuming that $X$ is the cotangent bundle of a smooth $G$-manifold.

\section{The A-model on the cotangent bundle of a $G$-manifold}
\label{Section 7}
In this section, we consider the A-model on $(X, \omega) := (T^*M, -d\theta)$, where $M$ is a smooth manifold and $\theta$ is the Liouville form on $T^*M$. Except in Subsection \ref{Subsection 7.1}, we always assume that $M$ is equipped with a smooth $G$-action $\rho_M$, inducing a Hamiltonian $G$-action on $(X, \omega)$.\par
The main part of this section are Subsections \ref{Subsection 7.3} and \ref{Subsection 7.4}, in which we will give an interpretation of equivariant Dolbeault cohomology and Guillemin-Sternberg theorem \cite{GuiSte1982} via sheaf cohomologies (\ref{Equation 5.2}) on intersections of $G$-invariant A-branes on $X$.\par
These $G$-invariant A-branes come from the conormal construction of Lagrangian submanifolds of $X$ and the construction of twisted cotangent bundles. In Subsection \ref{Subsection 7.1}, we will combine these two constructions to produce non-space filling non-Lagrangian A-branes. In Subsection \ref{Subsection 7.2}, we will study their $G$-invariance, as a preparation for Subsections \ref{Subsection 7.3} and \ref{Subsection 7.4}.

\subsection{Non-Lagrangian A-branes from vector bundles over submanifolds}
\quad\par
\label{Subsection 7.1}
Let $Z$ be a smooth submanifold of $M$. Recall that, for any flat Hermitian vector bundle $V_Z$ over $Z$, the conormal bundle $N^*Z$ equipped with the pullback of $V_Z$ is a Lagrangian A-brane $\mathcal{B}_{Z, V_Z}^{\operatorname{L}}$ on $(X, \omega)$. Now, we introduce a hybrid of this construction and the construction of twisted cotangent bundles \cite{Don2002} so as to produce (not necessarily Lagrangian) A-branes on $(X, \omega)$ supported on
\begin{equation*}
	C := T^*M \vert_Z.
\end{equation*}
We can see that $C$ is a coisotropic submanifold of $(X, \omega)$ by the following reason. Let $\theta'$ be the Liouville form on $Y$ and $\omega' = -d\theta'$. Notice that there is a canonical short exact sequence
\begin{equation*}
	\begin{tikzcd}
		N^*Z \ar[r] & C \ar[r, "p"] & Y := T^*Z
	\end{tikzcd}
\end{equation*}
of vector bundles over $Z$. As $p^*\theta' = \theta \vert_C$, $p^*\omega' = \omega \vert_C$. Therefore, $(TC)^{\perp \omega} = \ker dp \subset TC$.\par
To construct an A-brane supported on $C$, we need to obtain a transverse holomorphic symplectic form $\Omega$ on $C$ such that $\operatorname{Im} \Omega = \omega \vert_C$. One way is to choose a complex structure $J$ on $Z$. Then $J$ induces a complex structure $I_{\operatorname{can}}$ on $Y$ and an $I_{\operatorname{can}}$-holomorphic $(1, 0)$-form $\Theta_{\operatorname{can}}$ on $Y$ such that
\begin{itemize}
	\item the restriction of $\Theta_{\operatorname{can}}$ on each fibre of the bundle projection $\pi': Y \to Z$ is zero;
	\item $\Omega_{\operatorname{can}} := -d\Theta_{\operatorname{can}}$ is an $I_{\operatorname{can}}$-holomorphic symplectic form on $Y$; and
	\item the imaginary part of $\Theta_{\operatorname{can}}$ is $\operatorname{Im} \Theta_{\operatorname{can}} = \theta'$.
\end{itemize}
In particular, fibres of $\pi': Y \to Z$ are $\Omega_{\operatorname{can}}$-Lagrangian submanifolds of $Y$ and $\operatorname{Im} \Omega_{\operatorname{can}} = \omega'$. We can then obtain a desirable $2$-form on $C$ by pulling back $\Omega_{\operatorname{can}}$ along $p: C \to Y$. Furthermore, we can deform $\Omega_{\operatorname{can}}$ by a $d$-closed $(1, 1)$-form on $(Z, J)$. Let $\pi: X \to M$ be the bundle projection.

\begin{proposition}
	\label{Proposition 7.1}
	Suppose that $Z$ admits a complex structure $J$ and a Hermitian $J$-holomorphic vector bundle $E_Z$ over it with Chern connection $\nabla^{E_Z}$. Define
	\begin{equation}
		\label{Equation 7.1}
		C := T^*M \vert_Z, \quad E := (\pi \vert_C)^* E_Z \quad \text{and} \quad \nabla := (\pi \vert_C)^*\nabla^{E_Z} + \sqrt{-1} p^*(\operatorname{Re} \Theta_{\operatorname{can}}) \cdot \operatorname{Id}_E.
	\end{equation}
	Then $\mathcal{B}_{Z, E_Z}^{\operatorname{C}} = (C, E, \nabla)$ is a coisotropic A-brane on $(X, \omega)$.
\end{proposition}
\begin{proof}
	Let $\omega_Z = F_{\operatorname{tr}}^{\nabla^{E_Z}}$. Note that $\omega_Z$ is a $d$-closed real $2$-form of type $(1, 1)$ on $Z$. By Theorem 3.3 in \cite{BogDeeVer2022} (see also \cite{Don2002}), there exists a unique complex structure $I$ on $Y$ such that
	\begin{itemize}
		\item $\Omega := \Omega_{\operatorname{can}} + (\pi')^*\omega_Z$ is an $I$-holomorphic symplectic form on $Y$;
		\item $\pi': (Y, I) \to (Z, J)$ is holomorphic; and
		\item fibres of $\pi'$ are $\Omega$-Lagrangian submanifolds of $Y$.
	\end{itemize}
	Since $\pi \vert_C = \pi' \circ p$, $(E, \nabla)$ is the pullback of $(E', \nabla')$ along $p$, where
	\begin{equation*}
		E' := (\pi')^*E_Z \quad \text{and} \quad \nabla' := (\pi')^*\nabla^{E_Z} + \sqrt{-1} \operatorname{Re} \Theta_{\operatorname{can}} \cdot \operatorname{Id}_{E'}.
	\end{equation*}
	By direct computations, $\omega' = \operatorname{Im} \Omega$, $F_{\operatorname{tr}}^{\nabla'} = \operatorname{Re} \Omega$ and $F_0^{\nabla'} = (\pi')^* F_0^{\nabla^{E_Z}}$. Then the kernel of $F_{\operatorname{tr}}^\nabla = p^*\operatorname{Re} \Omega$ is $(TC)^{\perp \omega}$, and $F_{\operatorname{tr}}^\nabla + \sqrt{-1} \omega \vert_C = p^*\Omega$ is a transverse holomorphic symplectic form on $C$. As $F_0^{\nabla^{E_Z}}$ is of type $(1, 1)$, so is $F_0^{\nabla'}$. Therefore, $F_0^\nabla = p^*F_0^{\nabla'}$ is of transverse type $(1, 1)$.
\end{proof}

\begin{remark}
	In general, $(Y, I)$ and $(Y, I_{\operatorname{can}})$ are not biholomorphic, unless $\omega_Z = d\alpha_Z$ for some $(1, 0)$-form $\alpha_Z$ on $Z$. For more details, readers can refer to \cite{Lis1985}.
\end{remark}

\subsection{$G$-invariant A-branes on the cotangent bundle of a $G$-manifold}
\label{Subsection 7.2}
\quad\par
Recall that the $G$-action $\rho_M$ on $M$ naturally induces a $G$-action $\rho$ on $X$ such that $\pi: X \to M$ is $G$-equivariant. Indeed, $\rho$ is a Hamiltonian $G$-action on $(X, \omega)$ with a moment map $\mu: X \to \mathfrak{g}^*$ defined by $\langle \mu, a \rangle = \iota_{\chi(a)} \theta$ for all $a \in \mathfrak{g}$, where $\chi$ is the infinitesimal $\mathfrak{g}$-action on $X$ induced by $\rho$.\par
Suppose that $Z$ is a smooth submanifold of $M$ and $V_Z$ is a Hermitian vector bundle over $Z$ with a flat unitary connection. We can easily see that, if $Z$ is $G$-invariant and $V_Z$ is $G$-equivariant, then the Lagrangian A-brane $\mathcal{B}_{Z, V_Z}^{\operatorname{L}}$ is $G$-invariant with respect to the Hamiltonian $G$-action on $(X, \omega)$. Now we discuss $G$-invariance of the A-brane $\mathcal{B}_{Z, E_Z}^{\operatorname{C}}$ appeared in Subsection \ref{Subsection 7.1}.

\begin{proposition}
	Let $Z$ be a $G$-invariant smooth submanifold of $M$. Suppose that
	\begin{enumerate}
		\item $Z$ admits a complex structure and $E_Z$ is a Hermitian holomorphic vector bundle over $Z$ with Chern connection $\nabla^{E_Z}$;
		\item the induced $G$-action $\rho_Z$ on $Z$ is holomorphic and $(E_Z, \nabla^{E_Z})$ is $G$-equivariant.
	\end{enumerate}
	Then the coisotropic A-brane $\mathcal{B}_{Z, E_Z}^{\operatorname{C}} = (C, E, \nabla)$ defined as in (\ref{Equation 7.1}) is $G$-invariant.
\end{proposition}
\begin{proof}
	The coisotropic submanifold $C = T^*X \vert_Z$ is clearly $G$-invariant and $\rho_Z$ induces a $G$-action $\rho'$ on $Y = T^*Z$. There are naturally induced $G$-actions on $E$ and $E' = (\pi')^*E_Z$ such that all the arrows in the pullback diagram below are $G$-equivariant:
	\begin{center}
		\begin{tikzcd}
			E \ar[r] \ar[d] & E' \ar[r] \ar[d] & E_Z \ar[d]\\
			C \ar[r, "p"'] & Y \ar[r, "\pi{'}"'] & Z
		\end{tikzcd}
	\end{center}
	Since $(E_Z, \nabla^{E_Z})$ is $G$-equivariant, so is $(E', (\pi')^*\nabla^{E_Z})$.\par
	We know that the $G$-action $\rho'$ on $Y$ preserves $\operatorname{Im} \Theta_{\operatorname{can}}$. Since $G$ acts on $Z$ holomorphically, $\rho'$ is $I_{\operatorname{can}}$-holomorphic. Noting that $\Theta_{\operatorname{can}}$ is $I_{\operatorname{can}}$-holomorphic, $\rho'$ also preserves $\operatorname{Re} \Theta_{\operatorname{can}}$. Therefore, $(E', \nabla')$ is $G$-equivariant. By the same reason as above, $(E, \nabla)$ is $G$-equivariant.
\end{proof}

Finally, we compute the moment trace of $\mathcal{B}_{Z, E_Z}^{\operatorname{C}}$. Define $\omega_Z = F_{\operatorname{tr}}^{\nabla^{E_Z}}$. There is a unique $G$-equivariant smooth map $\mu_Z: Z \to \mathfrak{g}^*$ such that
\begin{itemize}
	\item if $a \in \mathfrak{g}$, $x \in Z$ and $u \in T_xZ$, then $\omega_Z(\chi_Z(a)(x), u) = \langle d\mu_Z(u), a \rangle$; and
	\item if $a \in \mathfrak{g}$ and $s \in \Gamma(Z, \det E_Z)$, then $\mathcal{L}_a^{\det E_Z} s = \nabla_{\chi_Z(a)}^{\det E_Z} s - \sqrt{-1} (\operatorname{rank} E_Z) \cdot  \langle \mu_Z, a \rangle s$.
\end{itemize}
Here, $\chi_Z$ is the infinitesimal $\mathfrak{g}$-action on $Z$ and $\nabla^{\det E_Z}$ is the connection on $\det E_Z$ induced by $\nabla^{E_Z}$. Let $\chi'$ be the infinitesimal $\mathfrak{g}$-action on $Y$. Define the map $\nu_{\operatorname{can}}: Y \to \mathfrak{g}^* \otimes \mathbb{C}$ by
\begin{equation}
	\label{Equation 7.2}
	\langle \nu_{\operatorname{can}}, a \rangle = \iota_{\chi'(a)} \Theta_{\operatorname{can}} \quad \text{for all } a \in \mathfrak{g}.
\end{equation}

\begin{proposition}
	The moment trace of the $G$-invariant coisotropic A-brane $\mathcal{B}_{Z, E_Z}^{\operatorname{C}}$ is
	\begin{equation*}
		\nu := p^*\nu_{\operatorname{can}} + (\pi \vert_C)^*\mu_Z: C \to \mathfrak{g}^* \otimes \mathbb{C}.
	\end{equation*}
\end{proposition}
\begin{proof}
	The connection $\nabla^{\det E}$ on $\det E$ induced by $\nabla$ is given by
	\begin{equation*}
		\nabla^{\det E} = (\pi \vert_C)^*\nabla^{\det E_Z} + \sqrt{-1} (\operatorname{rank} E) \cdot p^*(\operatorname{Re} \Theta_{\operatorname{can}}).
	\end{equation*}
	As $E$ and $E_Z$ are of the same rank, for all $a \in \mathfrak{g}$ and $s \in \Gamma(X, E)$,
	\begin{align*}
		\mathcal{L}_a^E s = & ((\pi \vert_C)^*\nabla^{\det E_Z})_{\chi(a)} s - \sqrt{-1} (\operatorname{rank} E_Z) \cdot \langle (\pi \vert_C)^* \mu_Z, a \rangle s\\
		= & \nabla_{\chi(a)}^{\det E} s - \sqrt{-1} (\operatorname{rank} E) \cdot \langle \operatorname{Re} \nu, a \rangle s.
	\end{align*}
	The equality on the second line is due to the fact that $\langle \operatorname{Re} \nu, a \rangle = \iota_{\chi(a)} p^* (\operatorname{Re} \Theta_{\operatorname{can}}) + \langle (\pi \vert_C)^*\mu_Z, a \rangle$. On the other hand, $\langle \operatorname{Im} \nu, a \rangle = \iota_{\chi(a)}p^*\operatorname{Im} \Theta_{\operatorname{can}} = \iota_{\chi(a)} \theta = \langle \mu, a \rangle$. We are done.
\end{proof}

It is shown in \cite{LerMonSja1993} that $\mu^{-1}(0) = \bigsqcup_{x \in M} N_x^*(G \cdot x)$. Moreover, $G$ acts on $\mu^{-1}(0)$ freely if and only if $G$ acts on $M$ freely, in which case $0$ is a regular value of $\mu: X \to \mathfrak{g}^*$ and
\begin{equation*}
	X \sslash G \cong T^*(M/G).
\end{equation*}
In the next subsections, we will discuss brane reduction of the $G$-invariant A-branes $\mathcal{B}_{Z, V_Z}^{\operatorname{L}}, \mathcal{B}_{Z, E_Z}^{\operatorname{C}}$ in special cases and see how sheaf cohomologies (\ref{Equation 5.2}) for these pairs are related to equivariant Dolbeault cohomology and Guillemin-Sternberg theorem \cite{GuiSte1982}.

\subsection{Equivariant Dolbeault cohomology in terms of intersections of brane reductions}
\label{Subsection 7.3}
\quad\par
Consider the following situation:
\begin{itemize}
	\item Let $(M, J)$ be a (not necessarily compact) complex manifold acted by $G$ holomorphically such that the induced $(\mathfrak{g}_\mathbb{C}, G)$-action $(\widetilde{\chi}_M, \rho_M)$ on $M$ is free (see Definition \ref{Definition C.1}).
	\item Let $(L_M, \nabla^{L_M})$ be a $G$-equivariant flat Hermitian line bundle over $M$ such that $\nabla^{L_M}$ restricts to the infinitesimal $\mathfrak{g}$-action on $L_M$.
\end{itemize}
Then the quotient manifold $M_0 := M/G$ admits a unique transverse holomorphic structure $\mathcal{E}_0$ such that $dq_M: T^{0, 1}M \to q_M^*\mathcal{E}_0$ is a vector bundle isomorphism, where $q_M: M \to M_0$ is the quotient map (see Proposition \ref{Proposition C.2}). By Proposition \ref{Proposition 4.8}, $(L_M, \nabla^{L_M})$ descends to a Hermitian line bundle $(L_{M_0}, \nabla^{L_{M_0}})$. Note that $\nabla^{L_{M_0}}$ restricts to a flat $\mathcal{E}_0$-connection $d_{\mathcal{E}_0}^{L_{M_0}}$.\par
Our goal is to reformulate the following proposition in terms of sheaf cohomologies (\ref{Equation 5.2}) on intersections of A-branes by considering the $G$-invariant A-branes $\mathcal{B} = \mathcal{B}_{M, M \times \mathbb{C}}^{\operatorname{L}}$ and $\mathcal{B}_{\operatorname{cc}} = \mathcal{B}_{M, L_M}^{\operatorname{C}}$. 

\begin{proposition}
	\label{Proposition 7.5}
	There is a linear isomorphism:
	\begin{equation*}
		H_{\mathcal{E}_0}^*(M_0, L_{M_0}) \cong H_{\overline{\partial}}^{0, *}(M, L_M)^G.
	\end{equation*}
\end{proposition}
\begin{proof}
	This is a direct consequence of Proposition \ref{Proposition C.3}.
\end{proof}

$\mathcal{B}$ and $\mathcal{B}_{\operatorname{cc}}$ have a transverse intersection being the zero section of $T^*M$. We identify it as $M$. Indeed, $M$ has the unique adapted complex structure $T^{0, 1}M$, and $\underline{\operatorname{Hom}}(\mathcal{B}, \mathcal{B}_{\operatorname{cc}}) = (L_M, \nabla^{L_M})$. Also, $\mathcal{B}_{\operatorname{red}}$ is the Lagrangian A-brane $\mathcal{B}_{M_0, M_0 \times \mathbb{C}}^{\operatorname{L}}$ on $X \sslash G = T^*M_0$. To describe $(\mathcal{B}_{\operatorname{cc}})_{\operatorname{red}}$, we first let $\mathcal{F}_0 = \mathcal{E}_0 \cap \overline{\mathcal{E}_0} \cap TM_0$ and $T_{\operatorname{quo}}M_0 = TM_0/\mathcal{F}_0$.

\begin{proposition}
	\label{Proposition 7.6}
	The moment section of $\mathcal{B}_{\operatorname{cc}}$ is neat. Moreover, the support of $(\mathcal{B}_{\operatorname{cc}})_{\operatorname{red}}$ is $T_{\operatorname{quo}}^*M_0$.
\end{proposition}
\begin{proof}
	Let $I_{\operatorname{can}}, \Theta_{\operatorname{can}}$ be given as in Subsection \ref{Subsection 7.1} and $\nu_{\operatorname{can}}$ be given as in (\ref{Equation 7.2}) for $Z = M$ and $Y = X$. As $\nabla^{L_M}$ restrict to the infinitesimal $\mathfrak{g}$-action on $L_M$, $\nu = \nu_{\operatorname{can}}$. We claim that $\nu_{\operatorname{can}}^{-1}(0)$ is the conormal bundle of $\widehat{\mathfrak{g}}_M$, where $\widehat{\mathfrak{g}}_M$ is the image of the map $M \times \mathfrak{g}_\mathbb{C} \to TM$ given by $(x, a) \mapsto \widetilde{\chi}_M(a)(x)$ (see Appendix \ref{Appendix C}). Let us first compute $\nu_{\operatorname{can}}(\alpha)$ for any $x \in M$ and $\alpha \in T_x^*M$. Observe that for all $v \in T_\alpha X$, $\theta(v) = \alpha(d\pi(v))$ and $\Theta_{\operatorname{can}}(v) = \theta(I_{\operatorname{can}}(v)) + \sqrt{-1} \theta(v)$. Let $\chi_M$ be the induced infinitesimal $\mathfrak{g}$-action on $M$. By (\ref{Equation 7.2}), for all $a \in \mathfrak{g}$,
	\begin{equation}
		\label{Equation 7.3}
		\langle \nu_{\operatorname{can}}(\alpha), a \rangle = \alpha(J\chi_M(a)(x)) + \sqrt{-1} \alpha(\chi_M(a)(x)).
	\end{equation}
	It implies that $\nu_{\operatorname{can}}(\alpha) = 0$ if and only if $\alpha$ annihilates the fibre of $\widehat{\mathfrak{g}}_M$ over $x$. Therefore, our claim holds. Note that $(\mathfrak{g}_\mathbb{C}, G)$ also acts on $X$ freely. Fix $\alpha \in \nu_{\operatorname{can}}^{-1}(0)$. By Lemma \ref{Lemma 4.5}, the rank of $\nu_{\operatorname{can}}$ at $\alpha$ is equal to $2\dim G$. Also, we know that $\dim \nu_{\operatorname{can}}^{-1}(0) = \dim X - 2 \dim G$. Therefore, $0 \in \mathfrak{g}^* \otimes \mathbb{C}$ is a clean value of $\nu$. As $\mathcal{B}_{\operatorname{cc}}$ is of rank $1$, the moment section of $\mathcal{B}_{\operatorname{cc}}$ is then neat. Now, we can see from Proposition \ref{Proposition C.2} that $X_{\operatorname{red}} = \nu_{\operatorname{can}}^{-1}(0) / G \cong T_{\operatorname{quo}}^*M_0$.
\end{proof}

\begin{proposition}
	\label{Proposition 7.7}
	$\mathcal{B}_{\operatorname{red}}, (\mathcal{B}_{\operatorname{cc}})_{\operatorname{red}}$ have a clean intersection $M_0 := M/G$, and
	\begin{equation*}
		\underline{\operatorname{Hom}}(\mathcal{B}_{\operatorname{red}}, (\mathcal{B}_{\operatorname{cc}})_{\operatorname{red}}) = (L_{M_0}, \nabla^{L_{M_0}}).
	\end{equation*}
Also, $\mathcal{E}_0$ is the unique adapted transverse holomorphic structure on $M_0$.
\end{proposition}
\begin{proof}
	We can easily see from Proposition \ref{Proposition 7.6} that $M_0$ and $X_{\operatorname{red}}$ have a clean intersection $M_0$. By Proposition \ref{Proposition 6.1}, $\underline{\operatorname{Hom}}(\mathcal{B}_{\operatorname{red}}, (\mathcal{B}_{\operatorname{cc}})_{\operatorname{red}}) = (L_{M_0}, \nabla^{L_{M_0}})$. Evidently, the transverse holomorphic structure associated with $\mathcal{B}_{\operatorname{red}}$ is $TM_0 \otimes \mathbb{C}$. Now, we claim that $(TM_0 \otimes \mathbb{C}) \cap \mathcal{E}_{\operatorname{red}} = \mathcal{E}_0$, where $\mathcal{E}_{\operatorname{red}}$ is the transverse holomorphic structure associated with $(\mathcal{B}_{\operatorname{cc}})_{\operatorname{red}}$. If our claim holds, then by Remark \ref{Remark 5.3}, $\mathcal{E}_0$ is the unique adapted transverse holomorphic structure on $M_0$.\par
	Let $X^0 = \nu^{-1}(0)$. Denote by $(\rho, \widetilde{\chi})$ the induced $(\mathfrak{g}_\mathbb{C}, G)$-action on the complex manifold $(X, I_{\operatorname{can}})$. By Lemma \ref{Lemma 4.6} and Remark \ref{Remark 4.7}, the image $\mathfrak{g}_{X^0}$ of $X^0 \times \mathfrak{g}$ (resp. $\widehat{\mathfrak{g}}_{X^0}$ of $X^0 \times \mathfrak{g}_\mathbb{C}$) under the map $X \times \mathfrak{g}_\mathbb{C} \to TX$ given by $(x, a) \mapsto \widetilde{\chi}(a)(x)$ lies in $TX^0$, and we have a short exact sequence of complex vector bundles over $X^0$:
	\begin{center}
		\begin{tikzcd}
			\mathfrak{g}_{X^0} \otimes \mathbb{C} \ar[r] & \widehat{\mathfrak{g}}_{X^0} \otimes \mathbb{C} + T^{0, 1}X^0 \ar[r, "dq"] & q^*\mathcal{E}_{\operatorname{red}}
		\end{tikzcd}
	\end{center}
	Here, $q: X^0 \to X_{\operatorname{red}}$ is the quotient map and $T^{0, 1}X^0$ denotes the antiholomorphic tangent bundle of $X^0$ as a complex submanifold of $(X, I_{\operatorname{can}})$. On the other hand, by Proposition \ref{Proposition C.2}, the map $dq_M \vert_{T^{0, 1}M}: T^{0, 1}M \to q_M^*\mathcal{E}_0$ is a complex vector bundle isomorphism. Note that $q_M$ is the restriction of $q$ on the zero section $M$. Now, fix $x \in M$ and let $y = q(x) \in M_0$.\par
	Suppose that $v \in \mathcal{E}_{0, y}$. Evidently, $v \in T_yM_0 \otimes \mathbb{C}$. Pick $u \in T_x^{0, 1}M$ such that $dq(u) = v$. As the zero section $M$ is a complex submanifold of $X^0$, $u \in T_x^{0, 1}X^0$ and hence $v = dq(u) \in \mathcal{E}_{\operatorname{red}, y}$.\par
	Conversely, suppose that $v \in (T_yM_0 \otimes \mathbb{C}) \cap \mathcal{E}_{\operatorname{red}, y}$. Pick $u \in T_xM \otimes \mathbb{C}$ such that $dq(u) = v$. Also, pick $w \in T_x^{0, 1}X^0$ and $w', w'' \in \mathfrak{g}_{(x)} \otimes \mathbb{C}$ such that $dq(w + w' + Iw'') = v$. Note that $dq(w') = dq(w'') = 0$ and $-\sqrt{-1}w'' + Iw'' \in T_x^{0, 1}X^0$. Now we replace $w -\sqrt{-1}w'' + Iw''$ by $w$ and see that $dq(w) = v$. Then $dq(u - w) = 0$, whence $u - w \in \mathfrak{g}_{(x)} \otimes \mathbb{C} \subset T_xM \otimes \mathbb{C}$. It forces that $w \in (T_xM \otimes \mathbb{C}) \cap T_x^{0, 1}X^0 = T_x^{0, 1}M$ and therefore $v = dq(w) \in \mathcal{E}_{\operatorname{red}, y}$. Our claim holds.
\end{proof}

In conclusion, Proposition \ref{Proposition 7.5} can be reformulated as
\begin{equation*}
	H_{\mathcal{E}_0}^*(M_0, \underline{\operatorname{Hom}}(\mathcal{B}_{\operatorname{red}}, (\mathcal{B}_{\operatorname{cc}})_{\operatorname{red}})) \cong H_{T^{0, 1}M}^*(M, \underline{\operatorname{Hom}}(\mathcal{B}, \mathcal{B}_{\operatorname{cc}}))^G.
\end{equation*}

\subsection{Guillemin-Sternberg theorem in terms of intersections of brane reductions}
\label{Subsection 7.4}
\quad\par
Consider the following situation:
\begin{itemize}
	\item Let $(M, \omega_M)$ be a compact connected K\"ahler manifold acted by $G$ holomorphically.
	\item 
	Let $(L_M, \nabla^{L_M})$ be a $G$-equivariant prequantum line bundle of $(M, \omega_M)$. In particular, $F^{\nabla^{L_M}} = -\sqrt{-1}\omega_M$ and, as discussed in  Subsection \ref{Subsection 3.2}, the $G$-action on $(M, \omega_M)$ is Hamiltonian with a moment map $\mu_M$ induced by the $G$-equivariant structure on $(L_M, \nabla^{L_M})$. Assume that $G$ acts on $\mu_M^{-1}(0)$ freely.
\end{itemize}
Then we obtain the K\"ahler reduction $M \sslash G = \mu_M^{-1}(0)/G$, and $(L_M \vert_{\mu_M^{-1}(0)}, \nabla^{L_M} \vert_{\mu_M^{-1}(0)})$ descends to a Hermitian holomorphic line bundle $L_{M \sslash G}$ equipped with the Chern connection $\nabla^{L_{M \sslash G}}$.\par
Here, our goal is to reformulate the following theorem in terms of sheaf cohomologies (\ref{Equation 5.2}) on intersections of A-branes by considering the $G$-invariant A-branes $\mathcal{B} = \mathcal{B}_{M, M \times \mathbb{C}}^{\operatorname{L}}$ and $\mathcal{B}_{\operatorname{cc}} = \mathcal{B}_{M, L_M}^{\operatorname{C}}$.\par

\begin{theorem}[\cite{Bra2001, Tel2000}]
	\label{Theorem 7.8}
	There is a $\mathbb{C}$-linear isomorphism
	\begin{equation*}
		H_{\overline{\partial}}^{0, *}(M \sslash G, L_{M \sslash G}) \cong H_{\overline{\partial}}^{0, *}(M, L_M)^G.
	\end{equation*}
\end{theorem}

The above theorem refines Guillemin-Sternberg `quantization commutes with reduction' theorem \cite{GuiSte1982} to the level of individual cohomology groups. Related studies can be found in \cite{LeuWan2023, Mei1998, MouNunPerWan2024, TiaZha1998, Wan2024}.\par
To achieve our goal, we need to study the behaviour of the intersection of $\mathcal{B}_{\operatorname{red}}$ and $(\mathcal{B}_{\operatorname{cc}})_{\operatorname{red}}$, given that those brane reductions are well defined. Unfortunately, without the condition that $G$ acts on $M$ freely, we cannot obtain $X \sslash G$ as a smooth symplectic manifold in general. There are different approaches to describe $X \sslash G$ when it is singular, for instance stratified symplectic spaces \cite{LerMonSja1993, LerSja1991} and symplectic derived stacks \cite{AneCal2022, Ben2015, PatToeVaqVez2013}. For our purpose, in order to avoid unnecessary technicality caused by formulating the notion of A-branes on a `singular symplectic space', we make the following observations instead.\par
$\mathcal{B}$ and $\mathcal{B}_{\operatorname{cc}}$ have a transverse intersection $M$. We see that $\eta^{-1}(0) = \mu_M^{-1}(0)$, where $\eta$ is the map in (\ref{Equation 6.1}). Since $0 \in \mathfrak{g}^*$ is a regular value of $\nu_M$, $0 \in \mathfrak{g}^* \otimes \mathbb{R}^3$ is a clean value of $\eta$. Proposition \ref{Proposition 6.1} suggests that we should consider the smooth manifold $\eta^{-1}(0) / G = M \sslash G$ as `the intersection of the singular brane reductions of $\mathcal{B}$ and $\mathcal{B}_{\operatorname{cc}}$'.\par
A more precise way to describe $M \sslash G$ as the intersection of A-branes is as follows. We know that the holomorphic $G$-action on $M$ extends to $G_\mathbb{C}$-action on $M$. We take the stable locus $M'$ of $M$, which is open and dense in $M$. Hence,
\begin{equation*}
	X' := T^*M'
\end{equation*}
is an open dense subset of $X$ containing $\eta^{-1}(0)$ and $G$ acts on $X'$ freely. We can then consider the $G$-invariant A-branes $\mathcal{B}' = \mathcal{B}_{M', M' \times \mathbb{C}}^{\operatorname{L}}$ and $\mathcal{B}'_{\operatorname{cc}} = \mathcal{B}_{M', L_{M'}}^{\operatorname{C}}$ on $(X', \omega \vert_{X'})$. Here, $(L_{M'}, \nabla^{L_{M'}})$ is the restriction of $(L_M, \nabla^{L_M})$ onto $M'$. The supports of $\mathcal{B}'$ and $\mathcal{B}'_{\operatorname{cc}}$ are $M'$ and $T^*M'$ respectively, and they have a transverse intersection $M'$, which is a neighbourhood of $\eta^{-1}(0)$ in $M$.\par
Note that $X' \sslash G = T^*M_0$, where $M_0 = M' / G$, and $\mathcal{B}'_{\operatorname{red}} = \mathcal{B}_{M_0, M_0 \times \mathbb{C}}^{\operatorname{L}}$. The fact that the $G_\mathbb{C}$-action on $M'$ is locally free implies that $(\mathfrak{g}_\mathbb{C}, G)$ acts on $M'$ freely in the sense of Definition \ref{Definition C.1}. By Proposition \ref{Proposition C.2}, $M_0$ admits a unique transverse holomorphic structure $\mathcal{E}_0$ such that $dq_M: T^{0, 1}M' \to q_M^*\mathcal{E}_0$ is a complex vector bundle isomorphism, where $q_M: M' \to M_0$ is the quotient map. Let $\mathcal{F}_0 = \mathcal{E}_0 \cap \overline{\mathcal{E}_0} \cap TM_0$ and $T_{\operatorname{quo}}M_0 = TM_0/\mathcal{F}_0$.

\begin{proposition}
	The moment section of $\mathcal{B}'_{\operatorname{cc}}$ is neat. Also, the support of $(\mathcal{B}'_{\operatorname{cc}})_{\operatorname{red}}$ is an affine subbundle of $T^*M_0$ modelled by the vector bundle $T_{\operatorname{quo}}^*M_0$.
\end{proposition}
\begin{proof}
	Recall that $\nu = \nu_{\operatorname{can}} + \pi^*\mu_M$, where $\pi: X \to M$ is the bundle projection and $\nu_{\operatorname{can}}$ is given as in (\ref{Equation 7.2}) for $Z = M$ and $Y = X$. Fix $x \in M'$. We can see that the restriction
	\begin{equation*}
		\nu \vert_{T_x^*M'} = \nu_{\operatorname{can}} \vert_{T_x^*M'} + \mu_M(x): T_x^*M' \to \mathfrak{g}_\mathbb{C}
	\end{equation*}
	is an affine map. We claim that $\nu^{-1}(0) \cap T_x^*M' = \nu_{\operatorname{can}}^{-1}(-\mu_M(x)) \cap T_x^*M'$ is non-empty. As $(\mathfrak{g}_\mathbb{C}, G)$ acts on $M'$ freely, the linear map $\mathfrak{g}_\mathbb{C} \to T_xM'$ given by $a \mapsto \widetilde{\chi}_M(a)(x)$ is injective. Thus, there exists $\alpha \in T_x^*M'$ such that for all $a \in \mathfrak{g}$,
	\begin{equation*}
		\alpha(J\chi_M(a)(x)) = -\langle \mu_M(x), a \rangle \quad \text{and} \quad \alpha(\chi_M(a)(x)) = 0,
	\end{equation*}
	where $\chi_M$ is the infinitesimal $\mathfrak{g}$-action on $M$. This implies that $\alpha \in \nu^{-1}(0) \cap T_x^*M'$ by the formula (\ref{Equation 7.3}). Our claim holds. Thus, $\nu^{-1}(0) \cap X'$ is an affine subbundle of $T^*M'$ modelled by $\nu_{\operatorname{can}}^{-1}(0) \cap X'$.\par
	Similar to the proof of Proposition \ref{Proposition 7.6}, by counting the dimension of $\nu^{-1}(0) \cap X'$ and the rank of $L_{M'}$, we see that the moment section of $\mathcal{B}'_{\operatorname{cc}}$ is neat. Again, we know from the same proof that $(\nu_{\operatorname{can}}^{-1}(0) \cap X') / G \cong T_{\operatorname{quo}}^*M_0$. We are done.
\end{proof}

From the construction of $\mathcal{B}'_{\operatorname{cc}} = \mathcal{B}_{M', L_{M'}}^{\operatorname{C}}$ in Proposition \ref{Proposition 7.1}, we can see that the restriction of the Hermitian line bundle associated with $\mathcal{B}'_{\operatorname{cc}}$ onto $M'$ is $(L_{M'}, \nabla^{L_{M'}})$. On the other hand, $\mathcal{B}' = \mathcal{B}_{M', M' \times \mathbb{C}}^{\operatorname{L}}$ is equipped with the trivial flat Hermitian line bundle. Thus,
\begin{equation*}
	\underline{\operatorname{Hom}}(\mathcal{B}', \mathcal{B}'_{\operatorname{cc}}) = (L_{M'}, \nabla^{L_{M'}}).
\end{equation*}

\begin{proposition}
	$(\mathcal{B}')_{\operatorname{red}}, (\mathcal{B}'_{\operatorname{cc}})_{\operatorname{red}}$ have a clean intersection $M \sslash G$ and
	\begin{equation*}
		\underline{\operatorname{Hom}}((\mathcal{B}')_{\operatorname{red}}, (\mathcal{B}'_{\operatorname{cc}})_{\operatorname{red}}) = (L_{M \sslash G}, \nabla^{L_{M \sslash G}}).
	\end{equation*}
	Moreover, the underlying complex structure on the K\"ahler manifold $M \sslash G$ is adapted.
\end{proposition}
\begin{proof}
	We have seen that $0$ is a clean value of $\eta$. By Proposition \ref{Proposition 6.1}, $(\mathcal{B}')_{\operatorname{red}}$ and $(\mathcal{B}'_{\operatorname{cc}})_{\operatorname{red}}$ have a clean intersection $\eta^{-1}(0) / G = M \sslash G$, and $\underline{\operatorname{Hom}}(\mathcal{B}'_{\operatorname{red}}, (\mathcal{B}'_{\operatorname{cc}})_{\operatorname{red}}) = (L_{M \sslash G}, \nabla^{L_{M \sslash G}})$. Observe that $TM_0 \otimes \mathbb{C}$ is the transverse holomorphic structure associated with $\mathcal{B}'_{\operatorname{red}}$. Now, let $\mathcal{E}_{\operatorname{red}}$ be the transverse holomorphic structure associated with $(\mathcal{B}'_{\operatorname{cc}})_{\operatorname{red}}$.\par
	We claim that the intersection of $TM_0 \otimes \mathbb{C}$ and $\mathcal{E}_{\operatorname{red}}$ is a zero vector bundle. We omit a proof of this claim as it would be similar to that in Proposition \ref{Proposition 7.7}. But we point out a key difference from Proposition \ref{Proposition 7.7} that, in this case, as $M'$ is a totally real submanifold of $X'$, the intersection of $TM' \otimes \mathbb{C}$ and $T^{0, 1}X^0$ is a zero vector bundle. Eventually, because $L_{M \slash G}$ is a holomorphic line bundle over $M \sslash G$, the complex structure on $M \sslash G$ is adapted.
\end{proof}

In conclusion, Theorem \ref{Theorem 7.8} can be reformulated as

\begin{proposition}
	\label{Proposition 7.11}
	There is a $\mathbb{C}$-linear isomorphism
	\begin{equation*}
		H_{T^{0, 1}M_0}^*(M_0, \underline{\operatorname{Hom}}((\mathcal{B}')_{\operatorname{red}}, (\mathcal{B}'_{\operatorname{cc}})_{\operatorname{red}})) \cong H_{T^{0, 1}M}^*(M, \underline{\operatorname{Hom}}(\mathcal{B}, \mathcal{B}_{\operatorname{cc}}))^G.
	\end{equation*}
\end{proposition}

As we have seen in this subsection, while this paper focuses on reduction of coisotropic A-branes in a \emph{non-singular} sense, singular spaces naturally come into our discussion. In the future, we hope to use a stacky approach to formulate coisotropic A-branes, their intersections, and their brane reductions in a more general context (see a related work \cite{Qin2024}).

\appendix
\section{Linear algebra in transverse holomorphic symplectic geometry}
\label{Appendix A}
We state some basic linear algebras involved in this paper. Proofs are left for readers.\par
Let $V$ be a finite-dimensional $\mathbb{R}$-vector space and $\Omega: V \times V \to \mathbb{C}$ be a skew-symmetric $\mathbb{R}$-bilinear map. Let $W$ be an arbitrary $\mathbb{R}$-vector subspace of $V$, $W^{\perp \Omega}$ be the $\Omega$-orthogonal complement of $W$ in $V$ and $W_{\operatorname{quo}} := (W + V^{\perp \Omega}) / V^{\perp \Omega}$. Then $\Omega$ descends to a skew-symmetric $\mathbb{R}$-bilinear form $\widetilde{\Omega}: V_{\operatorname{quo}} \times V_{\operatorname{quo}} \to \mathbb{C}$. Suppose that there exists a complex structure $I$ on $V_{\operatorname{quo}}$ such that
\begin{equation*}
	\widetilde{\Omega}(Iu, v) = \sqrt{-1} \widetilde{\Omega}(u, v),
\end{equation*}
for all $u, v \in V_{\operatorname{quo}}$. Note that such a complex structure is unique. Denote the imaginary part of $\Omega$ by $\omega$ and let $W^{\perp \omega}$ be the $\omega$-orthogonal complement of $W$ in $V$. The condition on the complex structure $I$ implies that $\omega$ descends to a non-degenerate $\mathbb{R}$-bilinear form on $V_{\operatorname{quo}}$ and hence $V^{\perp \Omega} = V^{\perp_\omega}$. Similar statements hold for the real part of $\Omega$. 

\begin{proposition}
	\label{Proposition A.1}
	The following statements hold.
	\begin{itemize}
		\item If $W \subset U$ for a $\mathbb{R}$-vector subspace $U$ of $V$, then $U^{\perp \Omega} \subset W^{\perp \Omega}$. In particular, $V^{\perp \Omega} \subset W^{\perp \Omega}$.
		\item $(W^{\perp \Omega})_{\operatorname{quo}} = W^{\perp \Omega} / V^{\perp \Omega}$ is a $\mathbb{C}$-vector subspace of $V_{\operatorname{quo}}$.
		\item $W^{\perp \Omega} = \widehat{W}^{\perp \Omega}$, where $\widehat{W}$ is the preimage of $W_{\operatorname{quo}} + I W_{\operatorname{quo}}$ under the quotient map $V \to V_{\operatorname{quo}}$.
		\item If $V^{\perp \Omega} \subset W$ and $IW_{\operatorname{quo}} \subset W_{\operatorname{quo}}$, then $W^{\perp \Omega} = W^{\perp_\omega}$, $(W^{\perp \Omega})^{\perp \Omega} = W$ and
		\begin{equation*}
			\dim W + \dim W^{\perp \Omega} = \dim V + \dim V^{\perp \Omega}.
		\end{equation*}
	\end{itemize}
\end{proposition}

\section{Transverse holomorphic structures}
\label{Appendix B}
Let $C$ be a smooth manifold. Recall that there is a one-to-one correspondence between
\begin{itemize}
	\item a complex vector subbundle $\mathcal{E}$ of $TC \otimes \mathbb{C}$ such that $\mathcal{E} + \overline{\mathcal{E}} = TC \otimes \mathbb{C}$, and
	\item a vector subbundle $\mathcal{F}$ of $TC$ with an almost complex structure $I$ on $T_{\operatorname{quo}}C := TC / \mathcal{F}$.
\end{itemize}
Under this correspondence, $\mathcal{E}$ is the kernel of the canonical projection $TC \otimes \mathbb{C} \to T_{\operatorname{quo}}^{1, 0} C$, which is equal to the preimage of $T_{\operatorname{quo}}^{0, 1} C$ under the canonical projection $TC \otimes \mathbb{C} \to T_{\operatorname{quo}}C \otimes \mathbb{C}$. Moreover, $\mathcal{E} \cap \overline{\mathcal{E}} \cap TC = \mathcal{F}$. We say that $I$ is \emph{integrable} and $\mathcal{E}$ is a \emph{transverse holomorphic structure} on $C$ if $\mathcal{E}$ is involutive (which implies that $\mathcal{F}$ is also involutive).\par
From now on in this appendix, we suppose that $C$ is a transverse complex manifold.

\subsection{Transverse biholomorphisms and transversely holomorphic vector fields}
\quad\par
\begin{proposition}
	\label{Proposition B.1}
	Let $\phi: C \to C$ be a smooth map. Then the following conditions are equivalent.
	\begin{enumerate}
		\item \label{Proposition B.1 (1)} $d\phi: TC \otimes \mathbb{C} \to TC \otimes \mathbb{C}$ preserves $\mathcal{E}$.
		\item \label{Proposition B.1 (2)} $d\phi: TC \otimes \mathbb{C} \to TC \otimes \mathbb{C}$ preserves $\overline{\mathcal{E}}$.
		\item \label{Proposition B.1 (3)} $d\phi: TC \to TC$ descends to a vector bundle morphism $\widetilde{d\phi}: T_{\operatorname{quo}}C \to T_{\operatorname{quo}}C$ such that 
		\begin{equation}
			\label{Equation B.1}
			I \circ \widetilde{d\phi} = \widetilde{d\phi} \circ I,
		\end{equation}
	\end{enumerate}
\end{proposition}

\begin{definition}
	A \emph{transverse biholomorphism} of $C$ is a diffeomorphism $\phi$ of $C$ satisfying one of the equivalent conditions in Proposition \ref{Proposition B.1}.
\end{definition}

The following proposition shows that a diffeomorphism preserving a transverse holomorphic symplectic form is a transverse biholomorphism.

\begin{proposition}
	Let $(C, \Omega)$ be a transverse holomorphic symplectic manifold. If $\phi: C \to C$ is a diffeomorphism and $\phi^*\Omega = \Omega$, then $\phi$ is a transverse biholomorphism of $C$.
\end{proposition}
\begin{proof}
	Since $\phi^*\Omega = \Omega$, $\phi$ preserves $(TC)^{\perp \Omega}$. 
	In particular, 
	$d\phi: TC \to \phi^*(TC)$ descends to a vector bundle isomorphism $\widetilde{d\phi}: T_{\operatorname{quo}}C \to \phi^*(T_{\operatorname{quo}}C)$. Note that for all $x \in C$ and $u, v \in T_{\operatorname{quo}, x}C$,
	\begin{equation*}
		\widetilde{\Omega}(\widetilde{d\phi}(Iu), v) = \widetilde{\Omega}(Iu, \widetilde{d\phi}^{-1}(v)) = \sqrt{-1} \widetilde{\Omega}(u, \widetilde{d\phi}^{-1}(v)) = \sqrt{-1} \widetilde{\Omega}(\widetilde{d\phi}(u), v) = \widetilde{\Omega}(I\widetilde{d\phi}(u), v),
	\end{equation*}
	whence $\widetilde{d\phi} \circ I = I \circ \widetilde{d\phi}$. We are done.
\end{proof}

Now, we are concerned about vector fields on $C$ preserving its transverse holomorphic structure. For any $v \in \Gamma(C, TC)$, define $\widetilde{v} \in \Gamma(C, T_{\operatorname{quo}}C)$ by $\widetilde{v}(x) = v(x) + \mathcal{F}_x$ for all $x \in C$.

\begin{proposition}
	\label{Proposition B.4}
	Let $v \in \Gamma(C, TC)$. Then the following conditions are equivalent:
	\begin{enumerate}
		\item \label{Proposition B.4 (1)}
		For all $w \in \Gamma(C, \mathcal{E})$, $[v, w] \in \Gamma(C, \mathcal{E})$.
		\item \label{Proposition B.4 (2)}
		For all $w \in \Gamma(C, \overline{\mathcal{E}})$, $[v, w] \in \Gamma(C, \overline{\mathcal{E}})$.
		\item \label{Proposition B.4 (3)}
		For all $w_0, w_1 \in \Gamma(C, TC)$ such that $\widetilde{w_1} = I\widetilde{w_0}$, then $\widetilde{[v, w_1]} = I\widetilde{[v, w_0]}$.
	\end{enumerate}
\end{proposition}

When $(C, I)$ is a complex manifold (and hence $\mathcal{E} = T^{0, 1}C$), a real-valued vector field $v \in \Gamma(C, TC)$ satisfies one of the equivalent conditions in Proposition \ref{Proposition B.4} if and only if its $(1, 0)$-part $v^{1, 0} = \tfrac{1}{2}(v - \sqrt{-1} Iv)$ is holomorphic, i.e. $\overline{\partial} v^{1, 0} = 0$. This justifies the following definition.

\begin{definition}
	A real-valued vector field on $C$ is said to be \emph{transversely holomorphic} on $(C, \mathcal{E})$ if it satisfies one of the equivalent conditions in Proposition \ref{Proposition B.4}.
\end{definition}

Let $\Gamma_I(C, TC)$ be the space of transversely holomorphic real-valued vector fields on $(C, \mathcal{E})$ and $\Gamma_I(C, T_{\operatorname{quo}}C)$ be the image of $\Gamma_I(C, TC)$ under the canonical projection $\Gamma(C, TC) \to \Gamma(C, T_{\operatorname{quo}}C)$. Then we have a short exact sequence of vector spaces:
\begin{center}
	\begin{tikzcd}
		\Gamma(C, \mathcal{F}) \ar[r] & \Gamma_I(C, TC) \ar[r] & \Gamma_I(C, T_{\operatorname{quo}}C)
	\end{tikzcd}
\end{center}
The space $H_\mathcal{F}^0(C, T_{\operatorname{quo}}C)$ of flat sections in $\Gamma(C, T_{\operatorname{quo}}C)$ with respect to the $\mathcal{F}$-Bott connection inherits a Lie bracket such that if $u, v \in \Gamma(C, TC)$ and $\widetilde{u}, \widetilde{v} \in H_\mathcal{F}^0(C, T_{\operatorname{quo}}C)$, then $[\widetilde{u}, \widetilde{v}] = \widetilde{[u, v]}$.

\begin{lemma}
	We have $\Gamma_I(C, T_{\operatorname{quo}}C) \subset H_\mathcal{F}^0(C, T_{\operatorname{quo}}C)$.
\end{lemma}
\begin{proof}
	Fix $v \in \Gamma_I(C, TC)$ and $w \in \Gamma(C, \mathcal{F})$. As $w \in \Gamma(C, \mathcal{E})$, $[v, w] \in \Gamma(C, \mathcal{E})$; as $w \in \Gamma(C, \overline{\mathcal{E}})$, $[v, w] \in \Gamma(C, \overline{\mathcal{E}})$. Therefore, $[v, w] \in \Gamma(C, \mathcal{F} \otimes \mathbb{C})$. Since $v, w$ are real-valued, $[v, w] \in \Gamma(C, \mathcal{F})$.
\end{proof}

\begin{lemma}
	\label{Lemma B.7}
	$(\Gamma_I(C, T_{\operatorname{quo}}C), I, [\quad, \quad])$ is a complex Lie algebra.
\end{lemma}
\begin{proof}
	We first show that $I$ preserves $\Gamma_I(C, T_{\operatorname{quo}}C)$. Indeed, fix $v \in \Gamma_I(C, TC)$. Let $v' \in \Gamma(C, TC)$ be such that $\widetilde{v'} = I\widetilde{v}$. Then $v + \sqrt{-1}v' \in \Gamma(C, \mathcal{E})$. Suppose that $w \in \Gamma(C, \mathcal{E})$. Since $\mathcal{E}$ is involutive, $[v + \sqrt{-1}v', w] \in \Gamma(C, \mathcal{E})$. Since $v \in \Gamma_I(C, TC)$, $[v, w] \in \Gamma(C, \mathcal{E})$. Therefore, $[v', w] \in \Gamma(C, \mathcal{E})$. In conclusion, $v' \in \Gamma_I(C, TC)$. Eventually, by (\ref{Proposition B.4 (3)}) in Proposition \ref{Proposition B.4}, we can easily see that $(\Gamma_I(C, T_{\operatorname{quo}}C), I, [\quad, \quad])$ is a complex Lie algebra.
\end{proof}

Similarly, the space $H_\mathcal{E}^0(C, T_{\operatorname{quo}}^{1, 0} C)$ of flat sections in $\Gamma(C, T_{\operatorname{quo}}^{1, 0} C)$ with respect to the $\mathcal{E}$-Bott connection is a complex Lie algebra. Define a $\mathbb{C}$-linear isomorphism
\begin{equation*}
	(\Gamma(C, T_{\operatorname{quo}}C), I) \to (\Gamma(C, T_{\operatorname{quo}}^{1, 0} C), \sqrt{-1}), \quad \widetilde{v} \mapsto \widetilde{v}^{1, 0} := \tfrac{1}{2} (\widetilde{v} - \sqrt{-1} I \widetilde{v}).
\end{equation*}

\begin{proposition}
	The $\mathbb{C}$-linear map
	\begin{equation*}
		(\Gamma_I(C, T_{\operatorname{quo}}C), I, [\quad, \quad]) \to (H_\mathcal{E}^0(C, T_{\operatorname{quo}}^{1, 0} C), \sqrt{-1}, [\quad, \quad])
	\end{equation*}
	given by $\widetilde{v} \mapsto \widetilde{v}^{1, 0}$ is a complex Lie algebra isomorphism.
\end{proposition}
\begin{proof}
	By Lemma \ref{Lemma B.7}, we can see that the image of $\Gamma_I(C, T_{\operatorname{quo}}C)$ under the $\mathbb{C}$-linear map $\widetilde{v} \mapsto \widetilde{v}^{1, 0}$ lies in $H_\mathcal{E}^0(C, T_{\operatorname{quo}}^{1, 0} C)$. If $\widetilde{v}, \widetilde{w} \in \Gamma_I(C, T_{\operatorname{quo}}C)$, then
	\begin{align*}
		[\widetilde{v}^{1, 0}, \widetilde{w}^{1, 0}] = & \tfrac{1}{4} ([\widetilde{v}, \widetilde{w}] - [I\widetilde{v}, I\widetilde{w}] - \sqrt{-1}( [I\widetilde{v}, \widetilde{w}] - \sqrt{-1} [\widetilde{v}, I\widetilde{w}] ))\\
		= & \tfrac{1}{2} ( [\widetilde{v}, \widetilde{w}] - \sqrt{-1} I[\widetilde{v}, \widetilde{w}] ) = [\widetilde{v}, \widetilde{w}]^{1, 0}.
	\end{align*}
	The second equality is due to Proposition \ref{Proposition B.4} \eqref{Proposition B.4 (3)}. Thus, the map $\Gamma_I(C, T_{\operatorname{quo}}C) \to H_\mathcal{E}^0(C, T_{\operatorname{quo}}^{1, 0} C)$, $\widetilde{v} \mapsto \widetilde{v}^{1, 0}$ is a complex Lie algebra homomorphism. It remains to show that it is surjective. Let $u, u' \in \Gamma(C, TC)$ be such that $\widetilde{u} - \sqrt{-1}\widetilde{u'} \in H_\mathcal{E}^0(C, T_{\operatorname{quo}}^{1, 0} C)$. Fix $w \in \Gamma(C, \mathcal{E})$. Then $[u - \sqrt{-1}u', w] \in \Gamma(C, \mathcal{E})$ because $\widetilde{u} - \sqrt{-1}\widetilde{u'}$ is flat with respect to the $\mathcal{E}$-Bott connection. On the other hand, note that $u + \sqrt{-1}u' \in \Gamma(C, \mathcal{E})$. Thus, $[u + \sqrt{-1}u', w] \in \Gamma(C, \mathcal{E})$. It implies that $v := 2\widetilde{u} \in \Gamma_I(C, TC)$ and $\widetilde{u} - \sqrt{-1}\widetilde{u'} = \widetilde{v}^{1, 0}$. We are done.
\end{proof}

\subsection{Transversely holomorphic actions}
\quad\par

\begin{definition}
	A \emph{transversely holomorphic} $G$\emph{-action} on $C$ is a smooth $G$-action $\rho_C$ on $C$ such that for all $g \in G$, $\rho_C(g)$ is a transverse biholomorphism of $C$.
\end{definition}

Now, let $G$ act on $C$ by a transversely holomorphic $G$-action and $\chi_C$ be the induced infinitesimal $\mathfrak{g}$-action. Define a $\mathbb{C}$-linear map $\widetilde{\chi}_C: \mathfrak{g}_\mathbb{C} \to \Gamma(C, T_{\operatorname{quo}}C)$ as follows:
\begin{equation*}
	\widetilde{\chi}_C(a + \sqrt{-1} b) = \widetilde{\chi_C(a)} + I\widetilde{\chi_C(b)}, \quad \text{for all } a, b \in \mathfrak{g}.
\end{equation*}
Also note that the $G$-action on $\Gamma(C, TC)$ descends to a $G$-action on $\Gamma(C, T_{\operatorname{quo}}C)$, and $\Gamma_I(C, TC)$ and $\Gamma_I(C, T_{\operatorname{quo}}C)$ are $G$-subrepresentations of $\Gamma(C, TC)$ and $\Gamma(C, T_{\operatorname{quo}}C)$ respectively.

\begin{proposition}
	\label{Proposition B.10}
	For all $a \in \mathfrak{g}$, $\chi_C(a)$ is a transversely holomorphic real vector field on $C$. Moreover, the map $\widetilde{\chi}_C: \mathfrak{g}_\mathbb{C} \to \Gamma_I(C, T_{\operatorname{quo}}C)$ is a $G$-equivariant complex Lie algebra homomorphism.
\end{proposition}
\begin{proof}
	Clearly, for all $a \in \mathfrak{g}$ and $v \in \Gamma(C, \mathcal{E})$, $[\chi_C(a), v] \in \Gamma(C, \mathcal{E})$. As the canonical projection $\Gamma_I(C, TC) \to \Gamma_I(C, T_{\operatorname{quo}}C)$ and $\chi_C: \mathfrak{g} \to \Gamma_I(C, TC)$ are $G$-equivariant Lie algebra homomorphisms, so is their composition $\mathfrak{g} \to \Gamma_I(C, T_{\operatorname{quo}}C)$. Thus, this composition extends to a complex Lie algebra homomorphism $\widetilde{\chi}_C: \mathfrak{g}_\mathbb{C} \to \Gamma_I(C, T_{\operatorname{quo}}C)$. The $G$-actions on $\mathfrak{g}_\mathbb{C}$ and on $\Gamma_I(C, T_{\operatorname{quo}}C)$ commute with multiplication by $\sqrt{-1}$ and $I$ respectively. Therefore, $\widetilde{\chi}_C$ is $G$-equivariant.
\end{proof}

\section{Equivariant Dolbeault Cohomology for free holomorphic actions}
\label{Appendix C}
Let $(M, J)$ be a complex manifold acted by a holomorphic $G$-action $\rho_M$, and $(E, \nabla^E)$ be a $G$-equivariant Hermitian holomorphic vector bundle with Chern connection over $M$. We adopt the Cartan model for equivariant Dolbeault cohomology (see \cite{Tel2000}). The Dolbeault operator $\overline{\partial}^E$ on $E$ commutes with the $G$-action on $\Omega^*(M, E)$, thus preserves the $G$-invariant part $\Omega^{0, *}(M, E)^G$. We call $(\Omega^{0, *}(M, E)^G, \overline{\partial}^E)$ the $G$-\emph{equivariant Dolbeault complex} of $E$. Its cohomology is equal to $H_{\overline{\partial}^E}^{0, *}(M, E)^G$, known as the $G$-\emph{equivariant Dolbeault cohomology} of $E$.\par
By Proposition \ref{Proposition B.10}, the induced infinitesimal action $\chi_M: \mathfrak{g} \to \Gamma(M, TM)$ extends to a Lie algebra homomorphism $\widetilde{\chi}_M: \mathfrak{g}_\mathbb{C} \to \Gamma(M, TM)$. Note that $(\mathfrak{g}_\mathbb{C}, G)$ is a Harish-Chandra pair. We call the pair $(\widetilde{\chi}_M, \rho_M)$ the $(\mathfrak{g}_\mathbb{C}, G)$ -\emph{action on} $M$ \emph{induced by} the holomorphic $G$-action $\rho_M$.

\begin{definition}
	\label{Definition C.1}
	We say that $(\mathfrak{g}_\mathbb{C}, G)$ \emph{acts on} $M$ \emph{freely} if $G$ acts on $M$ freely and for all $x \in M$, the map $\mathfrak{g}_\mathbb{C} \to \widetilde{T_xM}$ given by $a \mapsto \widetilde{\chi}_M(a)(x)$ is injective.
\end{definition}

From now on, suppose that $(\mathfrak{g}_\mathbb{C}, G)$ acts on $M$ freely. Obviously, the image $\mathfrak{g}_M$ of $M \times \mathfrak{g}$ (resp. $\widehat{\mathfrak{g}}_M$ of $M \times \mathfrak{g}_\mathbb{C}$) under the map $M \times \mathfrak{g}_\mathbb{C} \to TM, (x, a) \mapsto \widetilde{\chi}_M(a)(x)$ is a vector subbundle of $TM$, and $\widehat{\mathfrak{g}}_M = \mathfrak{g}_M \oplus J\mathfrak{g}_M$. Let $M_0 = M / G$ and $q: M \to M_0$ be the quotient map.

\begin{proposition}
	\label{Proposition C.2}
	There exists a unique transverse holomorphic structure $\mathcal{E}_0$ on $M_0$ such that $dq \vert_{T^{0, 1}M}: T^{0, 1}M \to q^*\mathcal{E}^0$ is a complex vector bundle isomorphism. Also, $dq \vert_{T^{0, 1}M}: T^{0, 1}M \to \mathcal{E}^0$ is a Lie algebroid morphism and
	\begin{center}
		\begin{tikzcd}
			\mathfrak{g}_M \ar[r] & \widehat{\mathfrak{g}}_M \ar[r, "dq \vert_{\widehat{\mathfrak{g}}_M}"] & q^*\mathcal{F}_0
		\end{tikzcd}
	\end{center}
	is a short exact sequence of vector bundles over $M$, where $\mathcal{F}_0 = \mathcal{E}_0 \cap \overline{\mathcal{E}_0} \cap TM_0$.
\end{proposition}
\begin{proof}
	As $\mathfrak{g}_M \cap (J\mathfrak{g}_M)$ is a zero vector bundle, so is $(\mathfrak{g}_M \otimes \mathbb{C}) \cap T^{0, 1}M$. Hence, the following $G$-equivariant complex vector bundle morphism is injective:
	\begin{equation}
		\label{Equation C.1}
		dq \vert_{T^{0, 1}M}: T^{0, 1}M \to q^*(TM_0 \otimes \mathbb{C}),
	\end{equation}
	and its image is of the form $q^*\mathcal{E}_0$ for a unique complex vector subbundle $\mathcal{E}_0$ of $TM_0 \otimes \mathbb{C}$.\par
	Since $dq: TM \to q^*TM_0$ is surjective and $TM \otimes \mathbb{C} = T^{1, 0}M \oplus T^{0, 1}M$, $\mathcal{E}_0 + \overline{\mathcal{E}_0} = TM_0 \otimes \mathbb{C}$. Now, fix $v_1, v_2 \in \Gamma(M_0, \mathcal{E}_0)$. As (\ref{Equation C.1}) is a vector bundle isomorphism, we can pick $u_1, u_2 \in \Gamma(M, T^{0, 1}M)$ such that $u_1, u_2$ are $q$-related to $v_1, v_2$ respectively. Then $[u_1, u_2] \in \Gamma(M, T^{0, 1}M)$. As $[u_1, u_2]$ is $q$-related to $[v_1, v_2]$, $[v_1, v_2] \in \Gamma(M_0, \mathcal{E}_0)$. Therefore, $\mathcal{E}_0$ is a transverse holomorphic structure on $M_0$. The above argument also shows that $dq \vert_{T^{0, 1}M}: T^{0, 1}M \to \mathcal{E}_0$ is a Lie algebroid morphism.\par
	Finally, it is clear that the kernel of $dq \vert_{\widehat{\mathfrak{g}}_M}$ is $\mathfrak{g}_M$. Now, we compute its image. For $u, u' \in \mathfrak{g}_M$,
	\begin{equation*}
		dq(u + Ju') = dq(-\sqrt{-1}u' + Ju') = dq(\sqrt{-1}u' + Ju') \in \mathcal{E}_0 \cap \overline{\mathcal{E}_0} \cap TM_0 = \mathcal{F}_0.
	\end{equation*}
	Conversely, suppose that $y \in M_0$ and $v \in \mathcal{F}_{0, y}$. Pick $u \in TM$ such that $v = dq(u + \sqrt{-1}Ju)$. Then $v = dq(u)$ and $dq(Ju) = 0$, implying that $u \in J\mathfrak{g}_M \subset \widehat{\mathfrak{g}}_M$. We are done.
\end{proof}

Recall from Proposition \ref{Proposition 4.8} that if $\nabla^E$ restricts to the infinitesimal $\mathfrak{g}$-action on $E$, then $(E, \nabla)$ descends to a Hermitian vector bundle $(E_0, \nabla^{E_0})$ over $M_0$.

\begin{proposition}
	\label{Proposition C.3}
	Suppose that $\nabla^E$ restricts to the infinitesimal $\mathfrak{g}$-action on $E$. Then $\nabla^{E_0}$ restricts to a flat $\mathcal{E}_0$-connection $d_{\mathcal{E}_0}^{E_0}$ on $E_0$ and the following is a cochain isomorphism:
	\begin{equation}
		\label{Equation C.2}
		q^*: \left( \Gamma(M_0, \textstyle\bigwedge \mathcal{E}_0^* \otimes E_0), d_{\mathcal{E}_0}^{E_0} \right) \to (\Omega^{0, *}(M, E)^G, \overline{\partial}^E).
	\end{equation}
	In particular, $H_{\mathcal{E}_0}^*(M_0, E_0) \cong H_{\overline{\partial}^E}^{0, *}(M, E)^G$.
\end{proposition}
\begin{proof}
	Fix $y \in M_0$ and $v, v' \in \mathcal{E}_{0, y}$. Pick $x \in q^{-1}(y)$ and $u, u' \in T_x^{0, 1}M$ such that $v = dq(u)$ and $v' = dq(u')$. Then $F^{\nabla^{E_0}}(v, v') = F^{\nabla^E}(u, u') = 0$. It implies that $\nabla^{E_0}$ is flat along the $\mathcal{E}_0$-direction. Then, it follows from Proposition \ref{Proposition C.2} that (\ref{Equation C.2}) is a cochain isomorphism.
\end{proof}

\bibliographystyle{amsplain}
\bibliography{References}

\end{document}